\documentclass[10pt]{amsart}

\usepackage{amsmath}
\usepackage{amssymb}
\usepackage{yfonts}
\usepackage{latexsym}
\usepackage[english]{babel}
\usepackage{amsfonts}
\usepackage{verbatim}
\usepackage{mathrsfs}
\usepackage{amstext}
\usepackage{xspace}
\usepackage{amsrefs}
\usepackage{graphicx}
\usepackage{fontenc}
\usepackage[margin=0.75in]{geometry}  %%  for margin settings
\usepackage{hyperref} %% optional packages if hyperlinking is required in the
\usepackage{color}
\usepackage{enumerate}

\allowdisplaybreaks[1]

\newtheorem{theorem}{Theorem}[section]
\newtheorem{fact}[theorem]{Fact}
\newcommand{\bfact}{\begin{fact}}
\newcommand{\efact}{\end{fact}}

\def\carby{\subset\kern-3pt\subset}

\def\comment#1{}

\def\id{\kern.2em{\rm I}\kern-.56em{\rm 1}}

\def\subsetneq{\lower 2.2pt\hbox{$\,{\buildrel\subset\over\neq}\,$}}
\def\tto{\mapsto\kern-8pt\rightarrow}

\newcommand{\beq}{\begin{equation}}
\newcommand{\eeq}{\end{equation}}
\renewenvironment{proof}{\medskip{\sc Proof.}}{\proofend\medskip}

\newcommand{\ben}{\begin{enumerate}}
\newcommand{\een}{\end{enumerate}}
\newtheorem{proposition}[theorem]{Proposition}

\newtheorem{lemma}[theorem]{Lemma}
\newtheorem{corollary}[theorem]{Corollary}

\theoremstyle{definition}
\newtheorem{definition}[theorem]{Definition}
\newtheorem{example}[theorem]{Example}

\newtheorem{remark}[theorem]{Remark}

\title[Integrating brackets]{Integral representations for bracket-generating multi-flows}

\author{Ermal Feleqi}
\address{Dipartimento di Matematica ``Tullio Levi-Civita'', Universit{\`a} degli Studi di Padova, Via Trieste, 63, Padova 35121, Italy}
\email{feleqi@math.unipd.it}

\author{Franco Rampazzo}
\address{Dipartimento di Matematica ``Tullio Levi-Civita'', Universit{\`a} degli Studi di Padova, Via Trieste, 63, Padova 35121, Italy}
\email{rampazzo@math.unipd.it}

\begin{document}
\maketitle

\newcommand{\compo}{\circ}
\def\transv{{\cap\kern-5pt\lower -2.2pt\hbox{$|$}\kern-4.9pt\lower
-7pt\hbox{$-$}}}
\def\baba{|\kern-1.7pt|}
\def\stransv{{\cap\kern-5.3pt\lower -2.2pt\hbox{$\baba$}\kern-5.3pt\lower
-7pt\hbox{$-$}}}
\newcommand{\eps}{\varepsilon}
\newcommand{\beqn}{\begin{eqnarray}}
\newcommand{\eeqn}{\end{eqnarray}}
\def\doublearrow{\,\hbox to 19pt{$\longrightarrow$\hfill}
\kern-19pt\raise2.5pt\hbox to 19pt{$\longrightarrow$\hfill}}
\def\triplearrow{\,\,\raise-2.5pt\hbox to 19pt{$\longrightarrow$\hfill}
\kern-19pt\hbox to 19pt{$\longrightarrow$\hfill}
\kern-19pt\raise2.5pt\hbox to 19pt{$\longrightarrow$\hfill}}
\def\verylongrightarrow{\raise2.15pt\hbox to 28pt{\hrulefill}\kern-3.25pt\longrightarrow}
\def\longtriplearrow{\hbox to 44pt{$\verylongrightarrow$\hfill}
\kern-44pt\raise2.5pt\hbox to 44pt{$\verylongrightarrow$\hfill}
\kern-44pt\raise5pt\hbox to 44pt{$\verylongrightarrow$\hfill}}
\def\longtto{\mbox{$\verylongrightarrow\kern-28pt\verylongrightarrow$}}
\def\ovarrow#1{{\stackrel{#1}{\verylongrightarrow}}}
\def\ovaarrow#1{{\sta333333333333333ckrel{#1}{\tto}}}
\def\ovlarrow#1{{\stackrel{#1}{\verylongleftarrow}}}
\def\verylongleftarrow{\longleftarrow\kern-4pt\raise2.2pt\hbox to 25pt{\hrulefill}}
\def\longsearrow{\setlength{\unitlength}{0.5cm}
\begin{picture}(1.2,1.2)
\put(-0.2,1.2){\vector(1,-1){2}}
\end{picture}}
\def\longnearrow{\setlength{\unitlength}{0.5cm}
\begin{picture}(1.2,1.2)
\put(-0.2,-1.2){\vector(1,1){2}}
\end{picture}}
\def\longdownarrow{\setlength{\unitlength}{0.25cm}
\begin{picture}(0.6,1)
\put(0.2,1){\vector(0,-1){2}}
\end{picture}}
\def\longuparrow{\setlength{\unitlength}{0.25cm}
\begin{picture}(1.2,1.2)
\put(0.2,-0.5){\vector(0,1){2}}
\end{picture}}
\def\veceps{\vec{\eps}}
\def\vic{\pmb}
\def\Is{I^{\#}}
\def\Xis{\Xi^{\#}}
\def\xis{\xi^{\#}}
\def\etas{\eta^{\#}}
\def\zetas{\zeta^{\#}}
\def\qs{q^{\#}}
\def\ts{t^{\#}}
\def\taus{\tau^{\#}}
\def\nus{\nu^{\#}}
\def\restr{\,\lceil\,}
\def\corestr{\,\rceil\,}
\def\equaldef{{\buildrel \rm def \over =}}
\def\proofend{\hfill$\diamondsuit\,\,$}
\def\pmb#1{\setbox0=\hbox{$#1$}%
  \kern-.025em\copy0\kern-\wd0
  \kern.05em\copy0\kern-\wd0
  \kern-.025em\raise.0433em\box0 }
\def\pmbo#1{\setbox0=\hbox{$#1$}%
  \kern-.055em\copy0\kern-\wd0
  \kern.05em\copy0\kern-\wd0
  \kern-.055em\raise.0433em\box0 }
\def\bem{\em }
\def\bemu{{\bem U}}
\def\zz{\mathbb{Z}}
\def\qq{\mathbb{Q}}
\def\hh{\mathbb{H}}
\def\cc{\mathbb{C}}
\def\rr{\mathbb{R}}
\def\nn{\mathbb{N}}
\def\aa{\mathbb{A}}
\def\kk{\mathbb{K}}
\def\ii{\mathbb{I}}
\def\jj{\mathbb{J}}
\def\pp{\mathbb{P}}
\def\tpp{\tilde{\mathbb{P}}}
\def\vv{\mathbb{V}}
\def\ff{\mathbb{F}}
\def\tt{\mathbb{T}}
\def\dd{\mathbb{D}}
\def\gg{\mathbb{G}}
\def\ll{\mathbb{L}}
\def\pp{\mathbb{P}}
\def\bmath{\mathbb{B}}
\def\smath{\mathbb{S}}
\def\zmath{\mathbb{Z}}
\def\rrinf{\overline{\rr}}
\newcommand{\nball}{{\bar\bmath}^n}
\newcommand{\nballop}{\bmath^n}
\def\regu#1#2{{\rm CCA}(#1;#2)}
\let\d=\delta
\let\ve=\varepsilon
\let\l=\lambda
\let\z=\zeta
\let\o=\omega
\let\L=\Lambda
\let\s=\sigma
\let\Si=\Sigma
\let\r=\rho
\let\ol=\overline
\let \a =\alpha
\let \vf = \varphi
\newcommand{\rh}{\hat{\r}}
\newcommand{\rb}{\ol{\r}}
\newcommand{\Ec}{\mathcal{E}}
\newcommand{\Mcal}{\mathcal{M}}
\newcommand{\Fc}{{\mathcal F}}
\newcommand{\Ac}{{\mathcal A}}
\newcommand{\Uc}{{\mathcal U}}
\newcommand{\Hc}{{\mathcal H}}
\newcommand{\Gc}{{\mathcal G}}
\def\Tc{{\mathcal T}}
\newcommand{\be}{\beta}
\let\g=\gamma

\def\bel{\begin{equation}\label}
\def\eeq{\end{equation}}

\newcommand{\red}{\textcolor{black}}

\begin{abstract}
If $f_1,f_2$ are smooth vector fields on an open subset of an  Euclidean space and $[f_1,f_2]$ is their Lie bracket, the asymptotic formula \bel{asab}\Psi_{[f_1,f_2]}(t_1,t_2)(x)
- x =t_1t_2 [f_1,f_2](x) +o(t_1t_2),\eeq  where we have set $
\Psi_{[f_1,f_2]}(t_1,t_2)(x) \equaldef \exp(-t_2f_2)\circ\exp(-t_1f_1)\circ\exp(t_2f_2)\circ\exp(t_1f_1)(x)$, is valid for all $t_1,t_2$ small enough.
 In fact, the integral, exact formula
\bel{abstractform}
\Psi_{[f_1,f_2]}(t_1,t_2)(x)
- x = \int_0^{t_1}\int_0^{t_2}[f_1,f_2]^{(s_2,s_1)} (\Psi(t_1,s_2)(x))ds_1\,ds_2 ,
\eeq where $ [f_1,f_2]^{(s_2,s_1)}(y) \equaldef D\Big(\exp(s_1f_1)\circ \exp(s_2f_2{{)}}\Big)^{-1}\cdot [f_1,f_2](\exp(s_1f_1)\circ \exp(s_2f_2){(y)}), $ with ${{y = \Psi(t_1,s_2)(x)}}$ has also been proven. Of course \eqref{abstractform} can be regarded as an improvement of \eqref{asab}.  In this paper we show that an integral representation like \eqref{abstractform} holds true for any iterated {{bracket}} made {{from}} elements of a family of vector fields $\red{\{}f_1,\dots,f_{{k}}\red{\}}$. In perspective, these integral representations might  lie at the basis for extensions of asymptotic formulas involving {{nonsmooth}} vector fields.
\end{abstract}

%%%%%%%%%%%%%%% end abstract

\section{Introduction and preliminaries}

\subsection{A notational premise} Let us begin with  a few  notational conventions which are consistent with the so-called Agrachev-Gamkrelidze formalism (see \cite{AgGa78,AgGa80,RaSu07}). First, in the formulas involving flows and vector fields, we shall write the argument of a function on the left. For instance, if $M$ is a differentiable manifold,  $x\in M$  and $f$ is a {{locally Lipschitz}} vector field on $M$, we shall use $xf$ to denote the evaluation of $f$  at $x$. Similarly, for the value at $t$ of the Cauchy problem
$
\dot x = f(x){{,\,\,}} x(0)=\bar x
$ we shall write $\bar x e^{tf}$ (so in particular, the differential equation itself will be written $\frac{d}{dt} ({\bar x} e^{tf}) = \bar x e^{tf} f$).
Secondly, if $t\in \rr$, and  $f,g$ are $C^1$ vector fields, the notation $\bar xfe^{tg}$ stands for   the tangent vector at $\bar xe^{tg} $ obtained by {{(}}i) evaluating $f$ at $\bar x$ (so obtaining the vector $\bar x f$) and then {{(}}ii) by { mapping }  $\bar x f$ though the differential (at $\bar x$)  of the map   $x{{\mapsto}}  xe^{tg} $.  Finally, the  vector fields $f,g$ can be regarded as first order {{differential}} operators, so the  notation $fg$ reasonably stands for the second order {{differential}} operator which, in the conventional notation, would map any $C^2$ function $\phi$ to  $D(D\phi\cdot g)\cdot f${{.}}\footnote{In terms of Lie derivatives, this operator maps $\phi$ into $L_fL_g\phi${{.}}} In particular, the Lie  bracket $[f,g]$, which is a first order {{differential}} operator resulting as a difference between two second order {{differential}} operators, in this notation has the following expression: $[f,g]\equaldef fg-gf$.
These conventions turn out to be particularly convenient for the subject we are going to deal with. However, sometimes more conventional notation will be utilized as well and the context will be sufficient to avoid any confusion.

\subsection{The main question} Let $n$ be a positive integer and let  $M\subseteq \rr^n$ be an open subset. If $f_1,f_2$ are $C^1$ vector fields and $x\in M$, the  Lie bracket
$[f_1,f_2]$ verifies the well-known asymptotic formula
\bel{lie2as}
 x\Psi_{[f_1,f_2]}(t_1,t_2)=x+ t_1t_2 \cdot(x[f_1,f_2]) + o(t_1t_2),
\eeq
where we have set  $x\Psi_{[f_1,f_2]}(t_1,t_2){{\equaldef}}  xe^{t_1f_1}e^{t_2f_2}(e^{t_1f_1})^{-1}(e^{t_2f_2})^{-1}$ (=$xe^{t_1f_1}e^{t_2f_2}e^{-t_1f_1}e^{-t_2f_2}$). {{Note that}} \eqref{lie2as} is the same as \eqref{asab}, just rewritten in the above-introduced formalism).
Similarly, for a bracket of degree $3$ one has
\bel{lie3as}
 x\Psi_{[[f_1,f_2],f_3]}(t_1,t_2,t_3)=x+t_1t_2t_3\cdot (x[[f_1,f_2],f_3]) +o(t_1t_2t_3).
\eeq
where $ x\Psi_{[[f_1,f_2],f_3]}{{\equaldef}} x\Psi_{[f_1,f_2]}(t_1,t_2)e^{t_3f_3}\left(\Psi_{[f_1,f_2]}(t_1,t_2)\right)^{-1}(e^{t_3f_3})^{-1} {{.}}$
\,\,\,\footnote{Notice that the left-hand side can be written as the product of 10 (=4+1+4+1) flows:  $$  x\Psi_{[[f_1,f_2],f_3]}= xe^{t_1f_1}e^{t_2f_2}e^{-t_1f_1}e^{-t_2f_2}e^{t_3f_3}e^{t_2f_2}e^{t_1f_1}e^{{{-}}t_2f_2}e^{{{-}}t_1f_1}{{e^{-t_3f_3}.}} $$}
\if
In particular  \eqref{lie2as} can be used to    deduce the equivalence between  local commutativity and the vanishing of  $[f_1,f_2]$ in a neighborhood of $x$.
\fi
Asymptotic estimates like \eqref{lie2as}-\eqref{lie3as} can be utilized, through a suitable application of open mapping arguments, to deduce various controllability  results.

In this paper we aim at replacing  asymptotic estimates for multiflows  like the  above ones  with {\it integral, exact} formulas.  For a bracket of degree two such a formula has been provided in  \cite{RaSu01}. More precisely, if $f_1,f_2$ are vector fields of class $C^1$ then\ for every $t_1,t_2$ sufficiently small the equality
\bel{formulaRS1}x\Psi_{[f_1,f_2]}(t_1,t_2)
=x+\int_0^{t_1}\int_0^{t_2}x\Psi_{[f_1,f_2]}(t_1,s_2)[{f_1},{f_2}]^{({{s_2}},s_1)}\,ds_1\,ds_2\eeq
holds true, where we have set
\bel{intbr2int}\nonumber
[{f_1},{f_2}]^{({{s_2}},s_1)} \equaldef  e^{{{s_2}} {f_2}}e^{s_1
{f_1}} [{f_1},{f_2}]e^{-s_1
{f_1}} e^{-{{s_2}} {f_2}}{{.}}
\eeq
Formula \eqref{formulaRS1} says that the result of the flow  composition $x\Psi_{[f_1,f_2]}(t_1,t_2)$ can be calculated as the integral,    over the multi-time rectangle $[0,t_1]\times[0,t_2]$ of $ x\Psi_{{{[f_1,f_2]}}}(t_1,s_2) [f_1,f_2]^{(s_2,s_1)}$, namely the function that  maps each $(s_1,s_2)\in [0,t_1]\times[0,t_2] $ to the estimation at $x\Psi_{[f_1,f_2]}(t_1,s_2)$  of the {\it integrating} bracket $[{f_1},{f_2}]^{(s_2,s_1)}$.
Incidentally, let us observe that as a trivial byproduct of \eqref{formulaRS1} one gets  the commutativity theorem (stating that the flows of $f_1$ and $f_2$ locally commute if and only if $[f_1,f_2] \equiv 0$).

We shall construct integrating brackets corresponding to every iterated bracket so that formulas analogous to \eqref{formulaRS1} hold true. Though we will set our problem on an open subset of $\rr^n$, we  will perform such {{a}} construction in a chart invariant way, so that the resulting formulas are meaningful on a differentiable manifold as well.

Rather than stating here  the main theorem (see Theorem \ref{integralth} below), which would  require a certain number of technicalities, we limit ourselves  to illustrating the situation  in the  case of a degree $3$ bracket $[[f_1,f_2],f_3]$. Let us assume that $f_1$ and $f_2$ are of class $C^2$ and $f_3$ is of class $C^1$, and let us define the integrating bracket  $[[f_1,f_2],f_3]^{(t_1,{{s_3}},s_1,s_2)}$ by setting, for every $t_1,{{s_3}},s_1,s_2$ sufficiently small,
\begin{align*}
&[[f_1,f_2],f_3]^{(t_1,{{s_3}},s_1,s_2)}\equaldef \\
&(e^{{{s_3}}f_3}e^{t_1 f_1}e^{s_2f_2}e^{-t_1f_1}e^{-s_2 f_2})
 \Big[ (e^{s_2 f_2}e^{s_1f_1})  \left[f_1\,,\,f_2 \right](e^{s_2 f_2}e^{s_1f_1}) ^{-1}\,,\,  \,f_3 \Big]
(e^{{{s_3}}f_3}e^{t_1 f_1}e^{s_2f_2}e^{-t_1f_1}e^{-s_2 f_2})^{-1}{{.}}
\end{align*}
Then Theorem  \ref{integralth} says that  there exists  $\delta>0$ such that for all $t_1,t_2,t_3\in  [-\delta,\delta]$ %such that
\bel{lie3int}
x\Psi_{[[f_1,f_2],f_3]}(t_1,{{t_2}},t_3)=x+
\int_0^{t_1}\!\!\int_0^{t_2}\!\!\int_0^{t_3} x\Psi_{[[f_1,f_2],f_3]}(t_1,t_2,s_3)
[[f_1,f_2],f_3]^{(t_1,s_3,s_1,s_2)}
\,ds_1\,ds_2\,ds_3{{.}}
\eeq

Let us point out two main facts:  \begin{itemize}
\item[{{(}}i)]  {{O}}n one hand, formula \eqref{lie3int} is similar to  \eqref{formulaRS1}{{;}}
\item[{{(}}ii)] {{O}}n the other hand, there is a crucial difference in the definition of integrating bracket passing from the {{case of a}} degree $2$ {{bracket}} to {{the case of a}} degree {{greater than}} $2$ {{bracket}}; indeed while the integrating bracket$ [f_1,f_2]^{({{s_2}},s_1)}$  is defined as an integral (over $[0,t_1]\times [0,t_2]$) of a suitable  adjoint of the classical bracket at the points  $x\Psi_{[f_1,f_2]}(t_1,s_2)$,  the integrand in   $[[f_1,f_2],f_3]^{(t_1,{{s_3}},s_1,s_2)}$ contains the bracket of $[f_1,f_2]^{({{s_2}},s_1)}$ --instead of $[f_1,f_2]$-- and $f_3$.  In fact,  the definition  of {{a}} higher degree integrating bracket is given by induction and involves various adjoint{{s}} of classical brackets (see Definition~\ref{def-expbra}).

\end{itemize}

\subsection{A motivation} Integral representations may be regarded as improvements of asymptotic formulas. In fact, our interest for this issue was raised by  the  aim  of  laying down a basic setting  on which one can reasonably investigate families of  vector fields that are less regular than what is required by the classical definition of {{a}} (iterated) Lie bracket. A typical case where such an investigation might prove interesting is provided by the {Chow-Rashevsk{{i}} Theorem}, which for  $C^{\infty}$ vector fields  $f_1,\dots,f_k$, guarantees small-time local controllability at $x\in M$ for
driftless control systems {{of the form:}}
 $
\dot y = \sum_{i=1}^k u_if_i(y){{,}}
$
{{with}} $|u_i| \leq 1${{,}} as soon as a condition like $
Lie \{f_1,\dots,f_k\}(x)  = T_{x} M$
 is verified{{.}} \footnote{ $Lie \{f_1,\dots,f_k\}$ is the Lie algebra  generated by the family {{of vector fields}} $\red{\{}f_1,\dots,f_k\red{\}}${{.}}} Akin results are valid for vector fields $f_i$ of class  $C^{r_i}$, $r_i$ being the maximal order of differentiation needed to define all the (classical)  brackets that  make the Lie algebra rank condition to hold tru{{e}} %(\ref{fullrank})
(see  Subsection \ref{concluding-reg}).

 So, a  natural question might be the following: what about {{the}} Chow-Rashevsk{{i}} Theorem in the case when, sa{{y}} the vector fields  $f_1,\dots,f_k$ are such that each $f_i$, $i=1, \dots, {{k}}$, is  just of class $C^{r_i -1}$ with locally Lipschitz $r_i-1$-{{$th$}} order derivatives?
Some  different  answers have been proposed{{,}} e.g.{{,}} in {{\cite{MoMo14-Jacobi}, \cite{MoMo13-Involutive}}}, \cite{RaSu01}, \cite{SaWh06}.  %{\bf GNU mettere gli altri:}.
% qui si potrebbe mettere \cites{MoMo13-Jacobi, MoMo13-Involutive}
In particular, in \cite{RaSu01} a set-valued notion of bracket has been introduced for locally Lipschitz vector fields.  However, a  mere recursive definition of bracket of degree greater than two would not work (see{{,}} e.g.{{,}} \cite{RaSu07}*{Section~7}, where it is shown that such an iterated bracket would be {\it too small} for an asymptotic formula to hold true).  We think that the study of integral representations in the smooth case may represent a first step towards a useful definition of iterated bracket in the non smooth case (see Subsection \ref{Subsec-nonsmooth}).

\subsection{Outline of the paper} The paper is organized as follows: in the remaining part of the present section we recall  the concept of formal iterated bracket of letters $X_1, X_2, \dots$. In Section~2 we introduce the notion of {\it integrating bracket}{{.}}  Section~3 is devoted to the main result of the paper, namely Theorem \ref{integralth}, which provides exact representations for bracket-generating multi-flows through integrals involving integrating brackets. In  Section~4
we discuss the question of regularity in connection with the validity of integral formulas. As a byproduct of the main result we state a Chow-{{Rashevski}} theorem with low regularity assumptions. In Section~5 we provide a simple example remarking the crucial difference between integrating brackets of degree $2$ and those of higher degree.  Finally we discuss some motivations of the present article coming from the aim of extending asymptotic formulas (possibly, via {{the regularization of the {{vector fields}}}}) to a nonsmooth setting.

\subsection{Formal brackets}
 Given a fixed sequence ${\bf X}=(X_1,X_2,\ldots)$ of distinct
objects called {\em variables}, or {\em indeterminates},
let $W({\bf X})$ be the set of all words in the alphabet consisting {{of}}
$X_i$, the left bracket, the right bracket, and the comma. The
{\em bracket} of two members $W_1$, $W_2$ of $W({\bf X})$ is the word
$[W_1,W_2]$ obtained by writing first a left bracket, then $W_1$, then
a comma, then $W_2$, and then a right bracket. We call {\em iterated brackets}  of ${\bf X}$ the elements
of  the smallest subset $S$ of $W({\bf X})$ that contains the
single-letter words $X_j$ and is such that whenever $W_1$ and $W_2$
belong to $S$ it follows that $[W_1,W_2]\in S$.  The {\em degree} $deg(W)$  of a
word $W\in W({\bf X})$ is the length of the {\em letter sequence} of $W$, namely of the sequence $Seq(B)$ obtained from $W$ by deleting all the brackets and
commas. Clearly, if $W_1,W_2\in W({\bf X})$ then
$deg([W_1,W_2])=deg(W_1)+deg(W_2)$.

An iterated bracket $B\in ITB({\bf X})$ is {\em canonical} if
$Seq(B)=X_1X_2\cdots X_{deg(B)}$.
\if(For example,
$\,\,\,\,\,\,[[[X_1,X_2],X_3],[X_4,[X_5,X_6]]]\,\,\,\,$ is canonical,
while on the other hand the brackets
$\,\,\,\,\,\,[[[X_4,X_5],X_6],[X_7,[X_8,X_9]]]\,$,
$\,\,\,\,\,\,[[[X_1,X_2],X_3],[X_4,[X_5,X_7]]]\,$,
$\,\,\,\,\,\,[[[X_1,X_2],X_4],[X_3,[X_5,X_6]]]\,$, \ and
$[[[X_1,X_2],X_1],[X_3,[X_4,X_5]]]$ are not.)
\fi
{{G}}iven a canonical bracket $B\in ITB({\bf X})$ of $deg(B)=m$, and any
finite sequence $\sigma=(\sigma_1,\ldots,\sigma_n)$ of objects
(possibly with repetitions) such that $n\ge m$, we use $B(\sigma)$, or
$B(\sigma_1,\ldots,\sigma_m)$, to denote the expression obtained from
$B$ by substituting $\sigma_j$ for $X_j${{,}} $j=1,\ldots,m$. {{F}}or
example, (a)~if $B=[[X_1,X_2],[X_3,X_4]]$ then
$
B(f_1,f_2,g,h)=[[f_1,f_2{{]}},[g,h]],
B(f_1,f_2,f_1,f_2)=[[f_1,f_2{{]}},[f_1,f_2]{{]}}$,
(b)~if $B$ is any canonical bracket of degree $m$, then
\mbox{$B(X_1,X_2,\ldots,X_m)=B$}, (c)~if $B=[X_1,X_2]$ and ${\bf
f}=(f_1,f_2,f_3)$ then $B({\bf f})=[f_1,f_2]${{.}}

Given any canonical bracket $B$ of degree $m$ and any nonnegative
integer $\mu$, the {\em $\mu$-shift} of $B$ is the iterated bracket
$$
B^{(\mu)}=B(X_{1+\mu},X_{2+\mu},\ldots,X_{m+\mu})\,.
$$
{{F}}or \mbox{example}, if
\mbox{$\,\,\,\,B=[[[X_1,X_2],[X_3,[X_4,X_5]]],X_6]\,$}, then the
$4$-shift of $B$ is the bracket $B^{(4)}$ given by
$B^{(4)}=[[[X_5,X_6],[X_7,[X_8,X_9]]],X_{10}]\,${{.}}

A {\em semicanonical bracket} is an iterated  bracket $B$ which
coincides with  a $\mu$-shift of a \mbox{canonical} bracket for some nonnegative
integer $\mu$.
\if
 (For example, the bracket
$[[[X_4,X_5],X_6],[X_7,[X_8,X_9]]]$ is semicanonical, but
$\,\,[[[X_1,X_2],X_3],[X_4,[X_5,X_7]]]\,$,
$\,\,\,[[[X_1,X_2],X_4],[X_3,[X_5,X_6]]]\,$, \ and
$[[[X_1,X_2],X_1],[X_3,[X_4,X_5]]]$ are not.)
\fi

For every iterated bracket $B$ of degree $m>1$ there
exists a unique pair $(B_1,B_2)$ of brackets
such that $B=[B_1,B_2]$.
 The
pair $(B_1,B_2)$ is the {\em factorization} of $B$, and the brackets
$B_1$, $B_2$ are known, respectively, as the {\em left factor} and the
{\em right factor} of $B$.

If $B$ is semicanonical then both factors of $B$ are semicanonical as
well. If $B$ is canonical then the left factor of $B$ is canonical and
the right factor of $B$ {{i}}s semicanonical. Hence, if $B$ is canonical
of degree $m>1$ and $(B_1,\tilde B_2)$ is its factorization, there
exists a canonical bracket $B_2$ such that $\tilde
B_2=B_2^{(deg(B_1))}$, so that $B=[B_1,B_2^{(deg(B_1))}]$. We will
call the pair $(B_1,B_2)$ the {\em canonical factorization} of $B$.
{{F}}or example, if $B=[[X_1,X_2],[[X_3,X_4],X_5]]$, then the
factorization of $B$ is the pair
$\Big([X_1,X_2],[[X_3,X_4],X_5]\Big)$, and the canonical factorization
is the pair $([X_1,X_2],[[X_1,X_2],X_3])$.

\if If
$B=[[X_1,X_2],[X_3,X_4]]$, then the factorization of $B$ is the pair
$\Big([X_1,X_2],[X_3,X_4]\Big)$, and the canonical factorization is
the pair $\Big([X_1,X_2],[X_1,X_2]\Big)$.\fi

\medbreak\noindent  Let   $B=B_0^{(\mu)}$ be a semicanonical bracket, where  $B_0$ is  a canonical bracket of degree $m$. Let  $M$ be  a {{differentiable}} manifold and let ${\bf f}=(f_1,\ldots,f_{\nu})$ ($\nu\ge m+\mu$) be a finite sequence of vector fields on $M$. We use  $B({\bf f})$ to denote the expression obtained from $B$ by substituting $f_j$ for $X_{{j+\mu}}${{,}} $j=1,\dots,m$. If the manifold  $M$ and the vector fields $f_j$ are sufficiently regular, then we can regard
$B({\bf f})$ as an iterated  { Lie bracket}, in the {{classical}} sense. For instance, if $B= [[X_7,X_8],X_9]$ and ${\bf f} = (f,g,h,k{{)}}$ is a $4$-tuple of vector fields, then
$$ B({\bf f})= [[f,g],h] = [f,g]h - h[f,g] = fgh - gfh - hfg+ hgf.$$
Of course the regularity of the vector field $B({\bf f})$ depends on both the regularity of the {{vector}} fields $(f_1,\ldots,f_{\nu})$ and on the structure of $B$.

\section{Integrating  brackets}

\subsection{Bracket generating multi-flows}
 To simplify our exposition, when not otherwise specified {\it we shall assume the vector fields involved in the formulas are defined on a open subset $M\subseteq\rr^n$ and are  of class $C^\infty$}. However, the regularity question is obviously quite  important and will be treated in  Section~\ref{concludingsec}. In particular, vector fields will be assumed as regular as required by the structure of the involved formal brackets.

\begin{definition}\label{commutator}Let us  associate  with  a formal bracket  $B$ of degree $m$  and a $m$-tuple  ${\bf f}=(f_1,\dots,f_m)$ of vector fields  a  product $\Psi_B^{{\bf f}}(t_1,\ldots,t_m)$
of exponentials $e^{t_if_i}$, $i=1,\dots,m$. We proceed recursively:
\begin{itemize}\item[(i)]
 If $B=X_1$ (so that ${\bf f}$ consists of a single vector
field $f$)  we set
$$
\Psi_B^{{\bf f}}(t)\,\,\equaldef\,\,e^{tf}\,,
$$
i.e., for each $x\in M$ and each sufficiently small $t$,  $x\Psi_B^{{\bf f}}(t)= y(t)$ where (in the conventional notation)  $y(\cdot)$ is the solution to the Cauchy problem $\dot y =f(y),\,\,y(0)=x$.

\item[(ii)] If $deg(B)=m>1$ and $B=[B_1,B_2^{(m_1)} ]$ is the canonical
factorization of $B$, for  any
${\bf t}=(t_1,\ldots,t_m)$, we set
\begin{align*}
{\bf f}_{(1) }\equaldef(f_1,\ldots,f_{m_{1}}){{,}}\,\quad
{\bf f}_{(2) }\equaldef(f_{m_{1}+1},\ldots,f_{m}){{,}}\\
{\bf t}_{(1) }\equaldef(t_1,\ldots,t_{m_{1}}){{,}}\,\quad
{\bf t}_{(2) }\equaldef(t_{m_{1}+1},\ldots,t_{m}){{,}}
\end{align*}
and
$$
\Psi_B^{{\bf f}}(\bf t)\,\,\,\equaldef
\Psi_{B_1}^{{\bf
f}_{(1) }}({\bf t}_{(1) })\,\,
 \Psi_{B_2}^{{\bf f}_{(2) }}({\bf
t}_{{(2)}})\,
\left(\Psi_{B_1}^{{\bf
f}_{(1) } }({{{\bf t}_{(1)}}})\right)^{-1}\,
\, \left(\Psi_{B_2}^{{\bf f}_{( 2)} }({\bf
t}_{(2 )})\right)^{-1}.{{\footnotemark}}\,\,
\quad$$
\if
We call the so-obtained map ${\bf t}\mapsto\Psi_B^{{\bf f}}({\bf t})$ the {\it multiflow commutator} associated to $(B,{\bf f})$.
\fi
  \end{itemize}
  It is clear that  for every precompact subset $K\subset M$
there exist a neighborhood $U$ of $K$ and a $\delta>0$ such that  $ x\Psi_B^{{\bf f}}({\bf t})$ is defined for every $x\in U$ and ${\bf t}\in ]-\delta, \delta[^m$. However, when not otherwise stated,  we shall assume that vector fields $f_i$ are {\it complete}, meaning that  their flows $ (x, t) \mapsto x e^{t f_i}$  are well-defined for all $x\in M$ and $t\in \rr$. Obviously, the general case can be recovered by standard  ``cut-off function'' arguments{{.}} %(That is, one multiplies vector fields by cut-off functions, that is, compactly supported functions which $=1$ on their  the
\end{definition}

 \footnotetext{In \cite{BrBr11}, {{\cite{MoMo13-Almost-exponential}, \cite{MoMo12}}}, akin  maps, usually defined for brackets $B$ of the form  $[X_1,[X_2{{,\dots,}}[X_{m-1},X_m],{{\dots}}{{]]}}$ and   ${\bf t}$ %is replaced with
  of the form $(t,\dots,t)$, are called
  {\it quasiexponential}, {\it almost exponential}, or {\it approximate exponential maps}. }
 Let us illustrate the above definition  of $\Psi_B^{{\bf f}}({\bf t})$ by means of simple examples:
\begin{figure}[!ht]
\centering
\includegraphics[scale=0.2, width=6truecm,height=5truecm]{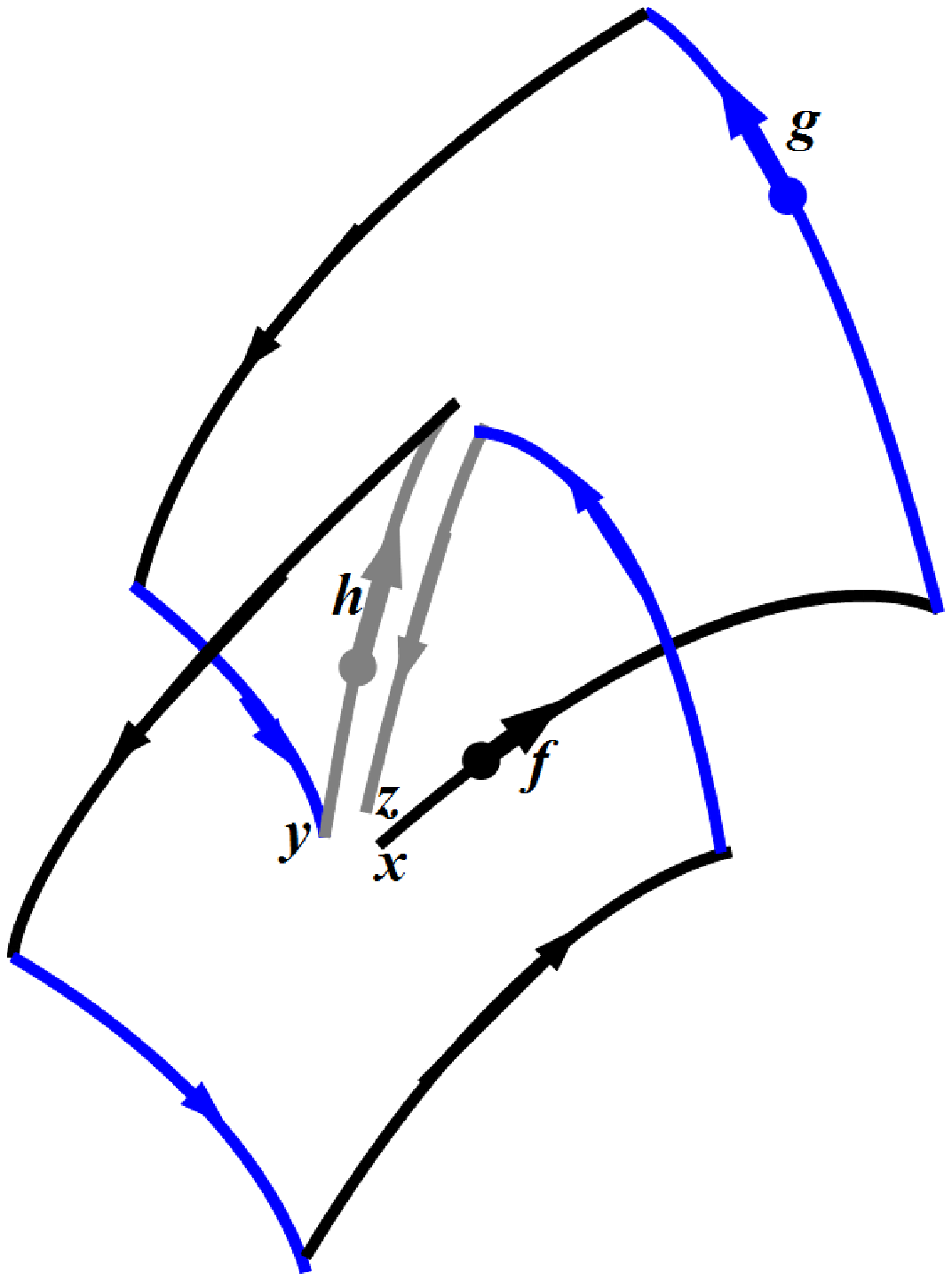}
\caption{$y= x\Psi_{[X_1,X_2]}^{(f,g)}(t_1,t_2),\,\,\, z= x\Psi^{(f,g,h)}_{[[X_1,X_2],X_3]}(t_1,t_2,t_3)${{.}}}
\end{figure}
\begin{enumerate}[(i)]
\item{} if $B=[X_1,X_2]$ and ${\bf f}=(f,g)$, then
$$
\Psi_B^{\bf f}(t_1,t_2)=e^{t_1f}e^{t_2g}e^{-t_1f}e^{-t_2g}\,;
$$

%\item{} if $B=[[X_1,X_2],X_3]$ and ${\bf f}=(f,g,h)$, then
%$$
%\Psi_B^{\bf f}(t_1,t_2,t_3)=e^{t_1f}e^{t_2g}e^{-t_1f}e^{-t_2g}
%e^{t_3h}e^{t_2g}e^{t_1f}e^{-t_2g}e^{-t_1f}e^{-t_3h}\,;
%$$
\item{} if $\textbf{}B=[X_1,[X_2,X_3]]$ and ${\bf f}=(f,g,h)$, then
\textbf{}$$
\Psi_B^{\bf f}(t_1,t_2,t_3)=e^{t_1f}
e^{t_2g}e^{t_3h}e^{-t_2g}e^{-t_3h}
e^{-t_1f}e^{t_3h}e^{t_2g}e^{-t_3h}e^{-t_2g}\,{{;}}
$$

\item{} if $\textbf{}B=[X_1,X_2],[X_3,X_4]]$ and ${\bf f}=(f,g,h,k)$, then
\textbf{}$$
\Psi_B^{\bf f}(t_1,t_2,t_3,t_4)=e^{t_1f}e^{t_2g}e^{-t_1f}e^{-t_2g} e^{t_3h}e^{t_4k}e^{-t_3h} e^{-t_4k}
e^{t_2g}e^{t_1f} e^{-t_2g}e^{-t_1f}e^{t_4k}e^{t_3h}e^{-t_4k}e^{-t_3h}
\,.
$$

\end{enumerate}

Observe that the number $N(B)$ of exponential factors of $\Psi_B^{\bf f}$ is given recursively by $N(B)=1$ if $deg(B) = 1$ and, for $m>1$,  $N(B) = 2(N(B_1) + N(B_2))$, where $[B_1,B_2]$ is the canonical factorization of $B$.

\subsection{Integrating  brackets}

  The {\it  integrating bracket} corresponding to $B$ and ${\bf f}$  will be defined as a $(2m-2)$-parameterized (continuous) vector field on $M$
 $$x\mapsto xB({\bf f})^{(t_1,\dots,t_{m_1-1},t_{m_1+1}, \dots{{,t_{m-1}, s_{m}}}, s_1,\dots,s_{m-1})}, $$ which, in particular,  (depends continuously on the  parameters $(t_1,\dots,t_{m_1-1},t_{m_1+1}, \dots{{,t_{m-1}, s_{m}}}, s_1\\,\dots,$ $s_{m-1})$) and verifies $ xB({\bf f})^{(0{{,}}\dots,0)} = xB({\bf f}) $. To begin with, let  us recall the notion of $Ad$ operator:
\begin{definition}  Let $U,V\subseteq M$ be open subsets and let   $\Phi:U\to V$ be a $C^r$ diffeomorphism ($r\geq 1$). If  $h$ is a vector field on ${{V}}$, $Ad_{\Phi} h$ is the vector field on $U$ defined by
$$
x\mapsto xAd_{\Phi} h \equaldef x\Phi h \Phi^{-1} \qquad \forall x\in U{{.}}
$$
\end{definition}
{{I}}n the conventional notation, the vector field  $Ad_{\Phi} h$ would be denoted by $x\mapsto {{D\Phi^{-1}|_{\Phi(x)}}}(h(\Phi(x)){{)}}$.
We remind that the $Ad$ operator is bracket preserving, namely
$$
Ad_{\Phi} [h_1,h_2] = [Ad_{\Phi} h_1, Ad_{\Phi} h_2],
$$
for all vector fields $h_1,h_2$.

Willing to define {\it integrating brackets} of degree greater than $2$, we cannot avoid  introducing a few more notation{{s}}. However, some examples following Definition \ref{def-expbra} should allow one to get  an intuitive idea of the bracket's construction.

If $d$ is any positive integer
 and $ \red{{\bf r} \in \rr^d} $ and  $\alpha\in \{1,\dots,d\}$, let us  use
$$\underset{\alpha}{\bf r}
$$
to denote the ${(d-1)}$-tuple obtained \red{from} $ {\bf r} $ by deleting the $\alpha$-{{$th$}} element. So, for instance,
if ${\bf r} = \red{(r_{21},r_{11},r_{40},r_{32}) \in \rr^4}$ one has
$$
\underset{1}{\bf r}= \red{(r_{11},r_{40},r_{32})}{{,}} \quad\underset{2}{\bf r}= \red{(r_{21},r_{40},r_{32})}, \quad 
\underset{3}{\bf r}= \red{(r_{21},r_{11},r_{32})}{{,}}\quad  \underset{4}{\bf r}= \red{(r_{21},r_{11},r_{40})}{{.}}
$$
For $\alpha,\beta\in  \{1,\dots,\red{d}\}$, $\alpha<\beta$, we also let $$\underset{\{\alpha,\beta\}}{\bf r}
 $$
 denote the ${(d-2)}$-tuple obtained by ${\bf r}$ by deleting the $\alpha$-{{$th$}} and $\beta$-{{$th$}} elements,
  so that, for instance, if ${\bf r} = \red{(r_{21},r_{11},r_{40},r_{32}) \in \rr^4}$,
$$
 %{\bf r}^{{2,4}}
  \underset{\{2,4\}}{\bf r} = \red{(r_{21},r_{40})}.
$$
When $d=1$,
we set
\bel{conv1}\nonumber %{\bf r}^{1}   
\underset{1}{ {\bf r}} 
\equaldef\emptyset.
\eeq
Also, if $d=2$, we set
\bel{conv2}\nonumber
 \underset{ \{ \alpha, \beta \} }{{\bf r}} \equaldef\emptyset.
\eeq

Let  $B$ be  an iterated bracket and let  $B=[B_1,B_2^{(m_1) }]$ be  its canonical factorization, with $deg (B) = m$, $deg (B_1) = m_1$, $deg( B_2)=m_2$, $m=m_1+m_2$. Let  ${\bf f} = (f_1,\dots,f_m)$ be  an $m$-tuple of vector fields .
We set, as before,
$$
{\bf t}=(t_1,\ldots,t_m){{,}}\qquad
{\bf t}_{(1) }\equaldef(t_1,\ldots,t_{m_{1}}){{,}}\,\qquad
{\bf t}_{(2) }\equaldef(t_{m_{1}+1},\ldots,t_{m}).
$$
Moreover, let
$${\bf s} = (s_1,\ldots,s_{m-1}{){,\qquad
{\bf s}_{(1) }\equaldef(s_1,\ldots,s_{m_{1}-1}),\,\qquad
{\bf s}_{(2) }\equaldef(s_{m_{1}+1},\ldots,s_{m-1}).}}$$

\begin{definition}[\bf Integrating bracket]\label{def-expbra}
We call    {\it integrating  bracket} ({\it corresponding to the pair $(B,{\bf f})$})  the $(2m-2)$-parameterized vector field  $ B ({\bf f})^{\left(\underset{{{\{m_1,m\}}}}{\bf t}{{, s_m}}, {\bf s}\right)}$  defined recursively as follows:
\begin{itemize}
\item[$\boxed{m=1}$]  If $m=1$ (so that $B=X_1$, ${\bf f} = f_1$),
we let
\bel{integratingone} B^{  \left(   \underset{{{\{m_1,m\}}}}{\bf t}{{, s_m}}, {\bf s}  \right)} ({\bf f}) = B^{\emptyset} ({\bf f}) \equaldef\displaystyle f_1. \eeq
\item[$\boxed{m>1}$] If $m=m_1+m_2\geq 2$, and  $B_1 = [B_{11}, B_{12}^{(m_{11}) } ]$,  $B_2 = [B_{21}, B_{22}^{(m_{21}) } ]$ are the canonical factorizations of $B_1$ and $B_2$, respectively, for some
$1\le m_{11} < m_1$, $1\le m_{21} < m_2 $,\begin{align}
   B ({\bf f})^{\left(\underset{{{\{m_1,m\}}}}{\bf t}{{, s_m}}, {\bf s}\right)} \equaldef  Ad_{\Psi_{B_2}^{{\bf f}_{(2)}}\left({{\underset{m-m_1}{{\bf t}_{(2)}}, s_m}}\right)\Psi_{B_1}^{{\bf f}_{(1)}}\left( \underset{{m_{1}}}{{\bf t}_{(1)}}  ,s_{m_1}\right) } \left[ \,B_1({\bf f}_{(1)})^{\left(\underset{\{{m_{11}},m_1\}}{{\bf t}_{(1)}},s_{m_1},{\bf s}_{(1)}\right)}\,,\,B_2 ({\bf f}_{(2)})^{\left(\underset{{{\red{\{m_{21},m-m_1\}}}}}{{{{\bf t}_{(2)}}}},\red{s_m,}{\bf s}_{(2)} \right) } \right]{{.}}\label{integratingnew}
     \end{align}\end{itemize}
 \end{definition}
 
 %{\color{blue} 
%\begin{definition}[\bf Integrating bracket]\label{def-expbra}
%We call    {\it integrating  bracket} ({\it corresponding to the pair $(B,{\bf f})$})  the $(2m-2)$-parameterized vector field  $ B ({\bf f})^{\left(\underset{{{\{m_1,m\}}}}{\bf t}{{, s_m}}, {\bf s}\right)}$  defined recursively as follows:
%\begin{itemize}
%\item[$\boxed{m=1}$]  If $m=1$ (so that $B=X_1$, ${\bf f} = f_1$),
%we let
%\bel{integratingone} B^{  \left(   \underset{{{\{m_1,m\}}}}{\bf t}{{, s_m}}, {\bf s}  \right)} ({\bf f}) = B^{\emptyset} ({\bf f}) \equaldef\displaystyle f_1. \eeq
%\item[$\boxed{m>1}$] If $m=m_1+m_2\geq 2$, and  $B_1 = [B_{11}, B_{12}^{(m_{11}) } ]$,  $B_2 = [B_{21}, B_{22}^{(m_{21}) } ]$ are the canonical factorizations of $B_1$ and $B_2$, respectively, for some
%$1\le m_{11} < m_1$, $1\le m_{21} < m_2 $,\begin{align}
%   B ({\bf f})^{\left(\underset{{{\{m_1,m\}}}}{\bf t}{{, s_m}}, {\bf s}\right)} \equaldef  Ad_{\Psi_{B_2}^{{\bf f}_{(2)}}\left({{\underset{m_2}{{\bf t}_{(2)}}, s_m}}\right)\Psi_{B_1}^{{\bf f}_{(1)}}\left( \underset{{m_{1}}}{{\bf t}_{(1)}}  ,s_{m_1}\right) } \left[ \,B_1({\bf f}_{(1)})^{\left(\underset{\{{m_{11}},m_1\}}{{\bf t}_{(1)}},s_{m_1},{\bf s}_{(1)}\right)}\,,\,B_2 ({\bf f}_{(2)})^{\left(\underset{{\{m_{21},m_2\}}}{{{{\bf t}_{(2)}}}},s_m,{\bf s}_{(2)} \right) } \right]{{.}}\label{integratingnew}
%     \end{align}\end{itemize}
% \end{definition}}
\begin{remark} When one of the indexes $m_1,m_2$ is equal one, formula \eqref{integratingnew} has  to be interpreted  as follows:
\begin{itemize}
 \item[${{\boxed{\underset{\displaystyle m_2=1}{\displaystyle m_1=1}}}}$] If $m_1=m_2=1$ (so $m=2$, $B=[X_1,X_2]$,  ${\bf f} = (f_1,f_2)${{)}},
  \begin{align}\nonumber B^{\left(\underset{{{\{m_1,m\}}}}{\bf t}{{, s_m}}, {\bf s} \right)} ({\bf f}) = [f_1, f_2]^{({{s_2}},s_1)}  \equaldef\displaystyle Ad_{e^{{{s_2}} f_2}e^{s_1f_1}} \left[f_1\,,\,f_2 \right]{{.}} \end{align}
\item[${{\boxed{\underset{\displaystyle m_2=1}{\displaystyle m_1>1}}}}$] If $m_1>1$ and $m_2=1$  (so $m_1+1 =m$, ${\bf f}_{(2)}= (f_m)$),
  \begin{align}\nonumber
 B^{\left(\underset{{{\{m_1,m\}}}}{\bf t}{{, s_m}}, {\bf s} \right)} ({\bf f})\equaldef  Ad_{e^{{{s_m}} f_m}\Psi_{B_{1}}^{{\bf f}_{(1)}}{{\left(\underset{m_1}{{\bf t}_{(1)}}, {s_{m_{1}}}\right)}}} \left[ \, B_1({\bf f}_{(1)})^{\left(\underset{\{{m_{11}},m_1\}}{{\bf t}_{(1)}},s_{m_1}, {\bf s}_{(1)} \right) }\,,  \,f_m \right]{{.}}
     \end{align}
\item[${{\boxed{\underset{\displaystyle m_2>1}{\displaystyle m_1=1}}}}$] If $m_1=1$ and $m_2>1$  (so $1+m_2 =m$, ${\bf f}_{(1) } = (f_1)$),

 \begin{align}\nonumber
    B^{\left(\underset{{{\{m_1,m\}}}}{\bf t}{{, s_m}}, {\bf s} \right)} ({\bf f})\equaldef  Ad_{\Psi_{B_{2}}^{{\bf f}_{(2)}}\left({{\underset{m-m_1}{{\bf t}_{(2)}}, s_m}}\right)e^{{{s_1}}f_1}} \left[f_1\,,B_2 ({\bf f}_{(2)})^{\left(\underset{{{{\red{\{m_{21},m-m_1\}}}}}}{{{{\bf t}_{(2)}}}},\red{s_m,}{\bf s}_{(2)} \right) }\right]{{.}}
     \end{align}

\end{itemize}
\end{remark}

\noindent{\bf Examples of integrating brackets: }\,\,

\begin{itemize}
\item[${{\boxed{\underset{\displaystyle m_2=1}{\displaystyle m_1=1}}}}$]

 If $B=[X_1,X_2]$, ${\bf f} =(f_1,f_2)$, then ($m=2$) {{and}}
\begin{align} 
&B ({\bf f})^{\left(\underset{{{\{m_1,m\}}}}{\bf t}{{, s_m}}, {\bf s}\right)}\nonumber\\
&= [f_1, f_2]^{( {{s_2}},s_1)}\nonumber\\
&= \displaystyle Ad_{e^{{{s_2}} f_2}e^{s_1f_1}} \left[f_1\,,\,f_2 \right]\nonumber\\
&= e^{{{s_2}} f_2}e^{s_1f_1} 
\left[f_1\,,\,f_2 \right]e^{-s_1f_1}e^{-{{s_2}} f_2}{{.}}
\end{align}

\item[${{\boxed{\underset{\displaystyle m_2=1}{\displaystyle m_1=2}}}}$]

 If $B=[[X_1,X_2],X_3]$, ${\bf f} =(f_1,f_2,f_3)$, (so ${\bf f}_{(1)}=(f_1,f_2)$ and ${\bf f}_{(2)}=(f_3))$, then
\begin{align} 
&B ({\bf f})^{\left(\underset{{{\{m_1,m\}}}}{\bf t}{{, s_m}}, {\bf s}\right)}\nonumber\\
&= [[f_1,f_2], f_3]^{(t_1,{{s_3}},s_1,s_2)}\nonumber\\
% B ({\bf f})^{(t_1,t_3,s_1,s_2)} =
&= Ad_{e^{{{s_3}} f_3}\Psi_{B_1}^{{{{\bf f}_{(1)}}}}(t_1,s_{2})} \left[[f_1,f_2]^{(s_2,s_1) }\,,\,  \,f_3  \right]\nonumber\\
&= Ad_{e^{{{s_3}}f_3}e^{t_1 f_1}e^{s_2f_2}e^{-t_1f_1}e^{-s_2 f_2}}
 \Big[ Ad_{e^{s_2 f_2}e^{s_1f_1}} \left[f_1\,,\,f_2 \right]  \,,\, f_3 \Big]\nonumber\\
&=e^{{{s_3}}f_3}e^{t_1 f_1}e^{s_2f_2}e^{-t_1f_1}e^{-s_2 f_2}\nonumber\\
&\quad\left[ \, e^{s_2 f_2}e^{s_1f_1}  \left[f_1\,,\,f_2 \right]e^{-s_1f_1}e^{-s_2 f_2} \,,\,  \,f_3 \right]\nonumber\\
&\quad e^{s_2 f_2}e^{t_1f_1}e^{-s_2f_2}e^{-t_1 f_1}e^{-{{s_3}}f_3}.\label{esempio2}
\end{align}
%$$
\item[${{\boxed{\underset{\displaystyle m_2=2}{\displaystyle m_1=2}}}}$]

 If $B=[[X_1,X_2],[X_3, X_4] ]$, ${\bf f} =(f_1,f_2,f_3,f_4)$, (so ${\bf f}_{(1)}=(f_1,f_2)$ and ${\bf f}_{(2)}=(f_3, f_4))$, then
\begin{align}
&B ({\bf f})^{\left(\underset{{{\{m_1,m\}}}}{\bf t}{{, s_m}}, {\bf s}\right)}\nonumber\\
&= [[f_1,f_2],[f_3, f_4] ]^{(t_1,t_3,{{s_4}},s_1,s_2,s_3)}\nonumber\\
&= Ad_{ e^{t_3 f_3}e^{{{s_4}}f_4} e^{-t_3 f_3}e^{-{{s_4}}f_4} e^{t_1f_1}e^{s_2f_2} e^{-t_1f_1}e^{-s_2f_2}}\Big[  Ad_{e^{s_2 f_2}e^{s_1f_1}}  \left[f_1\,,\,f_2 \right]\, ,\,  Ad_{e^{\red{s_4} f_4}e^{s_3f_3} } \left[f_3\,,\,f_4 \right] \Big]\nonumber\\
&=e^{t_3 f_3}e^{{{s_4}}f_4} e^{-t_3 f_3}e^{-{{s_4}}f_4} e^{t_1f_1}e^{s_2f_2} e^{-t_1f_1}e^{-s_2f_2}\nonumber\\
&\quad\Big[  e^{s_2 f_2}e^{s_1f_1}  \left[f_1\,,\,f_2 \right]e^{-s_1f_1}e^{-s_2 f_2}\red{\, ,\,}  e^{\red{s_4} f_4}e^{s_3f_3}  \left[f_3\,,\,f_4 \right]e^{-s_3f_3}e^{-\red{s_4} f_4} \Big]\nonumber\\
&\quad e^{s_2f_2}e^{t_1f_1} e^{-s_2f_2} e^{-t_1f_1} e^{{{s_4}}f_4} e^{t_3 f_3} e^{-{{s_4}}f_4} e^{-t_3 f_3}{{.}}
\end{align}
\end{itemize}
\begin{remark}{\rm  On one hand, we have made {{a small abuse}} of notation  by writing, for instance,
$ [[f_1,f_2],f_3]^{(t_1,{{s_3}},s_1,s_2)}$ instead of $B({\bf f})^{(t_1,{{s_3}},s_1,s_2)}$, with $B=[[X_1,X_2],X_3]$ and ${\bf f} = (f_1,f_2,f_3)$. We shall pursue with such notational simplifications whenever the danger of confusion is ruled out by the context.}
On the other hand let us point out that the definition of integrating bracket is based on the pair $(B,{\bf f})$ rather {{than}} on the vector field $B({\bf f})$. It may well happen that an integrating bracket $B^{{{\left(   \underset{{{\{m_1,m\}}}}{\bf t}{{, s_m}}, {\bf s} \right)}}} ({\bf f})$ of degree $m>2$ is  different from zero while the vector field $B({\bf f})$ (i.e.{{,}} the corresponding iterated Lie bracket) is identically equal to zero: see  Example~\ref{concluding-ex}.
\end{remark}

\subsection{Some basic properties of integrating brackets}

\begin{lemma}\label{lemma2} Let $f_1,f_2$ be  $C^2$ vector fields on $M$, and  let  $x\in M$ and $\delta_x>0$ such that
the integrating bracket  $x[f_1,f_2]^{{{(s_2}},s_1{{)}}}$ exists for   every $({{s_2}},s_1)\in [-\delta_x,\delta_x]^2$.\footnote{As remarked above, such a $\delta_x$ does exist, uniformly on precompact subsets of $M$.} Then
\bel{prop1}\nonumber x[f_1,f_2]^{({{s_2}},s_1)} = x[f_1, f_2] +  \int_0^{{{s_2}}} xAd_{e^{\tau f_2}}[f_2,[f_1,f_2]] d\tau  +     \int_0^{s_1}xAd_{e^{{{s_2}} f_2}e^{\sigma f_1}}[f_1,[f_1,f_2]] d\sigma{{.}}
\eeq
In particular, \bel{costante1}\nonumber x [f_1,f_2]^{({{s_2}},s_1)} = x[f_1, f_2]\qquad \forall x\in M,\quad \forall ({{s_2}},s_1)\in [-\delta_x,\delta_x]^2\eeq
if and only if \bel{nihilpotent}\nonumber
x[f_1,[f_1,f_2]] =  0 =  x[f_2,[f_1,f_2]]]\qquad \forall x\in M{{.\footnotemark}}\eeq
 \end{lemma}\footnotetext{Of course this condition is equivalent to the vanishing of  {\it all} brackets of degree $\geq 3${{.}}}
\begin{proof}
The Lemma is just an application to the $C^1$ map ${\Psi}:[-\delta_x,\delta_x]^2\to M$ of the following trivial fact:

{\it If ${\Psi}(0,0) = W\in\rr^n$, %for all $t_2\in [-\delta_x,\delta_x]$,
  then, for all $({{s_2}},s_1)\in [-\delta_x,\delta_x]^2$ one has
$$
{\Psi}({{s_2}},s_1) = W +\displaystyle  \int_0^{{{s_2}}} \frac{\partial {\Psi}(\tau,0)} {\partial \tau}\,d\tau +
\displaystyle  \int_0^{s_1}  \frac{\partial {\Psi}({{s_2}},\sigma)} {\partial \sigma}\,d\sigma{{.}}
$$}
In fact, setting $ {\Psi}({{s_2}},s_1)\equaldef x [f_1,f_2]^{({{s_2}},s_1)}$, one gets
$$
 \frac{\partial  {\Psi}({{\tau}},0)} {\partial \tau} =  xAd_{e^{\tau f_2}}[f_2,[f_1,f_2]]{{,}}\qquad
\frac{\partial  {\Psi}({{s_2}},\sigma)} {\partial \sigma}=xAd_{e^{{{s_2}} f_2}e^{\sigma f_1}}[f_1,[f_1,f_2]]{{.}}
$$
\end{proof}

\if
However  $[f_1,f_2]^{(t_2,s_1)} \neq [f_1, f_2]$ in general as the following example shows.
\begin{example}
In the Euclidean vector space $M = \rr^2$, consider the linear vector fields
\begin{gather}
 f_1 (x, y)=  A_1 \left(\begin{array}{c} x \\ y \end{array}  \right), \qquad    f_2 (x, y)=  A_2 \left(\begin{array}{c} x \\ y \end{array}
 \right)
\end{gather}
for all $(x, y) \in \rr^2$, where
\begin{gather}
A_1 = \left(
\begin{array}{cc}
 1 & 0 \\
 0 & 0 \\
\end{array}
\right),   \qquad
 A_2=
 \left(
\begin{array}{cc}
 0 & 1 \\
 1 & 0 \\
\end{array}
\right).
\end{gather}
Clearly $[f_1,f_2]$, $[f_1,f_2]^{(t_2,s_1)}$  are also linear with corresponding matrices
$[A_1, A_2]= A_2A_1-A_1A_2$, $[A_1, A_2]^{(t_2,s_1)} = e^{-t_2A_2}  e^{-s_1A_1} [A_1, A_2] e^{s_1A_1} e^{t_2A_2}$. It is easy to compute
\begin{gather}
[A_1, A_2]= \left(
\begin{array}{cc}
 0 & -1 \\
 1 & 0 \\
\end{array}
\right), \\ [A_1, A_2]^{(t_2,s_1)} =\left(
\begin{array}{cc}
 -\cosh (s_1) \sinh (2 t_2) & \sinh (s_1)-\cosh (s_1) \cosh (2 t_2) \\
 \cosh (s_1) \cosh (2 t_2)+\sinh (s_1) & \cosh (s_1) \sinh (2 t_2) \\
\end{array}
\right).
\end{gather}
Thus, in this case $[f_1,f_2]^{(t_2,s_1)} \neq [f_1, f_2]$.
\end{example}
\begin{corollary}[Antisymmetry]
Let  $f_1,f_2$ be as in Lemma~\ref{lemma2}, and assume that  { all} brackets of degree $\geq 3$ vanish identically. Then,
%if $B\equaldef [X_1,X_2]$ one has
%$$
%x[f_1,f_2] =xB_{(f_1,f_2)}{t_2,s_1}= -xB_{(f_2,f_1)}{t_2,s_1} =- x[f_2,f_1]
%$$
\[
x[f_1,f_2]=  x[f_1,f_2]^{(t_2, s_1)} = - x[f_2,f_1]^{(t_2, s_1)} = - x[f_2,f_1]
\]
for all $x\in M$ and
$|(t_2,s_1)|$ sufficiently small.
\end{corollary}

\fi

For  simplicity,  let us keep the notation
$$
\Psi(t_1,s_2)\equaldef \Psi_{[X_1,X_2]}^{(f_1,f_2)}(t_1,s_2){{.}}$$
As a consequence of Lemma \ref{lemma2}
 one gets:
\begin{proposition}\label{ad-br-cor} Consider vector fields $f_1,f_2$ of class  $C^3$.  Then, for every $x\in M$ and  every  vector field $f_3$ of class $C^1$ there exists $\delta_x>0$ such that  {{for every}} $(t_1,{{s_3}}, s_1,s_2)\in [-\delta_x,\delta_x]^4$, one has
\begin{align}
x[[f_1,f_2],f_3]^{(t_1,{{s_3}},s_1,s_2)} &= xAd_{e^{{{s_3}}f_3}\Psi(t_1,s_2)}\Big(\left[[f_1, f_2],f_3\right]+ \displaystyle
\int_0^{s_2}\Big[Ad_{ e^{\tau f_2}}[f_2,[f_1,f_2]], f_3\Big] d\tau\nonumber\\
&\quad +\displaystyle \int_0^{s_1}\Big[Ad_{e^{s_2 f_2}e^{\sigma f_1}}[f_1,[f_1,f_2]], f_3\Big] d\sigma\Big).\label{befor}
\end{align}

In particular, the following two statements are equivalent:
\begin{itemize}
\item[{{(i)}}]  For every vector field $f_3$ of class $C^1$ in a neighborhood of $x$, there is neighborhood $U$ of $x$ such that
\bel{self}y[[f_1,f_2],f_3]^{(t_1,{{s_3}},s_1,s_2)}
=
yAd_{e^{{{s_3}}f_3}\Psi(t_1,s_2)}\left[[f_1, f_2]\,,\,f_3\right]
\eeq
for all  $y\in U$ and all  $4$-tuples $(t_1,{{s_3}},s_1,s_2)$ sufficiently close to the origin.
\item[{{(ii)}}]  The identity
\bel{nihlpotent1}
y[f_1,[f_1,f_2]] =  0 =  y[f_2,[f_1,f_2]] \qquad
\eeq
holds true for every $y$ in a neighborhood of $x$.

\end{itemize}
\end{proposition}
\begin{proof}
%{\sc Proof.}
 To get \eqref{befor} it is sufficient to recall the definition
$$
[[f_1,f_2],f_3]^{(t_1,{{s_3}},s_1,s_2)}
  =Ad_{e^{{{s_3}}f_3}\Psi(t_1,s_2)}
 \left[[f_1,f_2]^{(s_2,s_1)}\,,\,  \,f_3 \right]\,$$
{{a}}nd to apply Lemma \ref{lemma2}.
Moreover, clearly \eqref{nihlpotent1} implies \eqref{self} for every $f_3$. To prove the converse claim, observe that by  \eqref{befor} and \eqref{self}, taking $(t_1,{{s_3}})=(0,0)$, one gets
\bel{f1}\nonumber
0 =  \displaystyle
 \int_0^{s_2}y\Big[Ad_{ e^{\tau f_2}}[f_2,[f_1,f_2]], f_3\Big] d\tau +\displaystyle
 \int_0^{s_1}y\Big[Ad_{e^{s_2 f_2}e^{\sigma f_1}}[f_1,[f_1,f_2]], f_3\Big] d\sigma\Big),
\eeq
for any $y$ in a neighborhood of $x$, for all vector fields $f_3$ of class  $C^1$ near $x$, and for  all $(s_1,s_2)$ sufficiently close to the origin. By computing  the partial derivatives at   $(s_1,s_2)=(0,0)$ of the right-hand side, in view of the continuity of integrands  one obtains
$$
y\Big[[f_2,[f_1,f_2]], f_3\Big] =0,\qquad y\Big[[f_1,[f_1,f_2]], f_3\Big] =0
$$
for all vector fields $f_3$ of class  $C^1$ near $x$ .
Then, necessarily, one has
$$y[f_2,[f_1,f_2]] =0,\qquad y[f_1,[f_1,f_2]] =0. $$
\end{proof}

\begin{remark} The fact that an integrating bracket  corresponding to a pair $(B,{\bf f})$, with  $deg( B)>2$,  {\it is not}, in general,  of the form $Ad_{\phi}B({\bf f})$ (where $\phi$ depends on $m$ parameters) marks a crucial difference with the case when $B=[X_1,X_2]$, for which, instead, one actually has $$[X_1,X_2]({\bf f})^{({{s_2}},s_1)} =Ad_{\phi}[f_1,f_2],$$ with $\phi = {{e^{{{s_2}}f_2} e^{s_1f_1}}}$. Incidentally, this fact has strong consequences in the attempt of defining a (set-valued)  Lie bracket" $[[f_1,f_2],f_3]$ when $f_1,f_2$ are of class $C^{1,1}$ and $f_3$ is merely  Lipschitz continuous  (see the Introduction and Section \ref{concludingsec}).
\end{remark}
\begin{remark} It is trivial  to check that   condition %$[f_2,[f_1,f_2]]$,
\eqref{nihlpotent1} remains   necessary for \eqref{self} to hold  even if the latter  is verified just for $n$ vector fields that are linearly independent at each $y\in U$.  Of course,   identity \eqref{self}  may well be  true for a particular $f_3$ even if \eqref{nihlpotent1} is not verified, as it is immediately apparent by taking  $f_3\equiv 0$, in which case \eqref{self}  holds with both sides vanishing.  However,  unless \eqref{nihlpotent1} is verified,   it is not true that the vanishing of  $[[f_1,f_2],f_3]$ implies the vanishing of  $[[f_1,f_2],f_3]^{{{(}}t_1,{{s_3}},s_1,s_2{{)}}}$, as shown in Example~\ref{concluding-ex} below. As a byproduct of Theorem~\ref{integralth} below, this is connected with (the almost obvious  fact) that in general the condition  $[[f_1,f_2],f_3]\equiv 0$   {\it  does not}  imply $\Psi_B^{{\bf f}} = Id_M$.
\end{remark}

\section{Integral representation }

We are now ready to state the main result. In this section, $m$ will  stand for a positive integer, $B$ will  represent  a formal bracket of degree $m$, $m_1$ will be the degree of the first bracket in the canonical decomposition of $B$,  and ${\bf f} = (f_1,\dots,f_m)$ will be {{an $m$-tuple of}} vector fields.

\begin{theorem}[\bf Integral representation]\label{integralth}For every $m$-tuple  ${\bf t}= (t_1,\dots,t_m)\in \rr^m$ one has
\bel{intform}
x\Psi_{ B}^{\bf f} ({\bf t}) = x + \int_0^{t_1}\cdots\int_0^{t_m} x\Psi_{ B}^{\bf f}\left(\underset{m}{\bf t},s_m\right) { B}({\bf f})^{\left(\underset{{{\{{m_{1}},m\}}}}{\bf t}{{,s_m}},{\bf{s}}\right)}ds_1\dots ds_m{{,}}\,\,\,
\eeq
where, according to the notation introduced in  previous section, we have set  $$\underset{m}{\bf t}=(t_1,\dots,\dots,t_{m-1}), \qquad \underset{{{\{{m_{1}},m\}}}}{\bf t} =(t_1,\dots,t_{m_1-1},t_{m_1+1},\dots,t_{m-1}){{,}}\qquad {\bf s} = (s_1,\dots,s_{m-1}) .$$
\end{theorem}
%\noindent {\bf Example}
\begin{example}
For brackets of degree $2, 3$,  see formulas \eqref{formulaRS1} and {{\eqref{lie3int}}}.

If $B=[[X_1,X_2],[X_3, X_4] ]$
(so ${\bf f}_{(1)}=(f_1,f_2)$ and ${\bf f}_{(2)}=(f_3, f_4))$,
then \eqref{intform} reads
\begin{align*}
&x\Psi_{ B}^{\bf f} ({\bf t}) \\
&= x e^{t_1{f_1}}e^{t_2{f_2}}e^{-t_1{f_1}}e^{-t_2 {f_2}} e^{t_3 {f_3}}e^{t_4 {f_4}}e^{-t_3
{f_3}}e^{-t_4 {f_4}}
e^{t_2 {f_2}}e^{t_1{f_1}}e^{-t_2
{f_2}}e^{-t_1{f_1}}e^{t_4 {f_4}}e^{t_3 {f_3}}e^{-t_4 {f_4}}e^{-t_3
{f_3}} \\
&= x +\displaystyle
\int_0^{t_1}\int_0^{t_2}\int_0^{t_3}\int_0^{t_4} x\Phi\Big[  e^{s_2f_2 }e^{s_1f_1}[f_1, f_2 ] e^{-s_1f_1} e^{-s_2f_2 }  \,,\,
e^{\red{s_4}f_4 }e^{s_3f_3}[f_3, f_4 ] e^{-s_3f_3} e^{-\red{s_4} f_4 } \Big]
\Gamma
  \,ds_1\,ds_2\,ds_3\,ds_4{{,}}
\end{align*}
where
$$
\Phi\equaldef
e^{t_1f_1} e^{t_2f_2}e^{-t_1f_1}e^{-t_2f_2}  e^{t_3f_3} e^{s_4f_4} e^{-t_3f_3}e^{-s_4f_4}  e^{t_2f_2}e^{t_1f_1}   e^{(s_2-t_2)f_2}  e^{-t_1f_1}  e^{-s_2f_2}
$$
and
$$
\Gamma\equaldef 
e^{s_2f_2}e^{t_1f_1} e^{-s_2f_2} e^{-t_1f_1} e^{{{s_4}}f_4} e^{t_3 f_3} e^{-{{s_4}}f_4} e^{-t_3 f_3}. 
%
% e^{s_4f_4}
% e^{t_3f_3}
% e^{-s_4f_4}
% e^{-t_3f_3}
%  e^{t_1f_1}  e^{t_2f_2}    e^{-t_1f_1}  e^{-t_2f_2}
%  e^{t_3f_3}
% e^{s_4f_4}
%e^{-t_3f_3}
%e^{-s_4f_4}
%e^{t_2f_2}  e^{t_1f_1}    e^{-t_2f_2}  e^{-t_1f_1}
$$

\end{example}
%\fi
%\begin{proof}
%Equality \eqref{intform} is a straightforward consequence of \eqref{intformV} and the following identity
%\[
%x\Psi_B^{{\bf f}}({\bf t}) -x =\int_0^{t_{m}} x\Psi_B^{{\bf f}}({\bf
%t}^{s_m}) V_B^{{\bf f}}({\bf t}^{s_m})\,ds_m\,,
%\]
%which in turn follows by the very definition of $V_B^{{\bf f}}$.
%\end{proof}
%\fi

 The proof of  Theorem \ref{integralth} will rely on an analogous result (see Theorem \ref{vectorth}    below)  concerning the  ($(t_1,\dots,t_{m})$-dependent)  {\it vector field }
$x\mapsto  xV_B^{\bf f}(t_1,\dots,t_{m})$ corresponding to the (${\bf t}$-dependent) {\it  local 1-parameter  action} $A$ defined as
\bel{action}\nonumber
(x,\tau)  \mapsto  A({\bf t},x,\tau) \equaldef x\left(\Psi_B^{{\bf
f}}({\bf t})\right)^{-1}\Psi_B^{{\bf
f}}\left(\underset{m}{\bf t},t_m+\tau\right){{.}}\quad
\eeq

\begin{definition}\label{def-V}  {{For}} every value of the parameter  ${\bf t}= (t_1,\dots,t_m) \in \rr^m$ let us define the vector field $x\mapsto xV_B^{\bf f}({\bf t})$ by setting, for every  $x\in M$,
 \begin{align}\label{d-V}
  xV_B^{\bf f}({\bf t})\equaldef\frac{\partial}{\partial \tau}{{\bigg|_{\tau=0}}}A({\bf t},x,\tau){{.}}{{\footnotemark}}
\end{align}
\end{definition}
\footnotetext{Notice that, for every  ${\bf t}$, $x\mapsto xV_B^{\bf f}({\bf t})$ is in fact a (intrisicly defined) vector field, for  $ A({\bf t},\cdot,\cdot)$ { is}   a true local {\it action}: this means that  $$ A({\bf t},x,0)=x  ,\qquad  A({\bf t},x,\tau_1+\tau_2) =  A({\bf t}, A({\bf t},x,\tau_1),\tau_2), $$
for all $\tau_1,\tau_2$ sufficiently small.}

%{\bf GNU: FORSE METTERE QUI UN DISEGNO}

\begin{theorem}\label{vectorth} Let $m$ be an integer {{greater than or equal to 1}}.
For every $m$-tuple ${\bf t}=(t_1,\dots,t_m)\in \rr^m$ one has
\bel{intformV}
xV_{B}^{\bf f} ({\bf t}) =  \int_0^{t_{1}}\cdots\int_0^{t_{m-1}} x{ B}({\bf f})^{\left(\underset{m_1}{\bf t},{\bf s} \right)}ds_1\dots ds_{m-1}\,\,\,,\eeq
where, as before,  $\underset{m_1}{\bf t}=(t_1,\dots,t_{m_1-1},t_{m_1+1},\dots,t_m)$, ${\bf s} = (s_1,\dots,s_{m-1}) .${{\footnotemark}}
\footnotetext{If $m=1$, so that $B =X_1$, ${\bf f} =f_1$, the formula above should be understood as  $xV_{B}^{\bf f}(t) = x B^t( {\bf f} ) = f_1$.}
\end{theorem}

{\it Proof of Theorem \ref{integralth}.}
Equality \eqref{intform} is a straightforward consequence of \eqref{intformV} and the following identity
\[
x\Psi_B^{{\bf f}}({\bf t}) -x =\int_0^{t_{m}} x\Psi_B^{{\bf f}}\left(\underset{m}{\bf
t}, {s_m}\right) V_B^{{\bf f}}\left(\underset{m}{\bf
t}, {s_m}\right)\,ds_m\,,
\]
which in turn follows by the very definition of $V_B^{{\bf f}}$. 
\proofend

  \subsection{Special cases of Theorem \ref{vectorth}} We postpone the general proof of   Theorem \ref{vectorth} to the next subsection and {\it let us treat directly the cases when $m=1,2,3$. } Actually the case when  $B= [[X_1,X_2], X_3]$ is a bit technical,  but,  still,  we prefer to perform all calculations for the simple reason that   they are paradigmatic of those needed in   the general proof.
\begin{itemize}

\item[{{$\boxed{m=1}$}}]

{\rm
The case when $m=1$  Theorem~\ref{vectorth} is trivial, since
\bel{exvf1}xV_{X}^{f}(t) = x f.\eeq
\item[{{$\boxed{m=2}$}}]

 In the case when  $m=2$  the proof of Theorem~\ref{vectorth} is  straightforward as well. Indeed
\begin{align}
&xV_{[X_1,X_2]}^{{{(}}f_1,f_2{{)}}}(t_1,t_2)\nonumber\\
&=\nonumber
 \dfrac{\partial}{\partial \tau}{{\bigg|_{\tau=t_2}}}\big(xe^{t_2{{f_2}}}e^{t_1f_1} e^{-t_2f_2}e^{-t_1f_1} e^{t_1f_1}e^{\tau f_2}  e^{-t_1f_1} e^{-\tau f_2}\big)\nonumber\\
 &= xe^{t_2f_2}\left(e^{t_1f_1}f_2 e^{-t_1f_1} - f_2\right)e^{-t_2 f_2}\nonumber\\
&= xe^{t_2f_2}\left(\int_0^{t_1}e^{\rho f_1}[f_1,f_2] e^{-\rho f_1}d\rho\right) e^{-t_2 f_2}\nonumber\\
&= \displaystyle \int_0^{t_1}  x[f_1\, f_2]^{(t_2,s_1)}ds_1.\label{exvf}
\end{align}}

\item[{{$\boxed{m=3}$}}]

The case when $B= [[X_1,X_2], X_3]$: we have to show that
\bel{exvf-3}
xV_{[[X_1,X_2], X_3] }^{(f_1,f_2, f_3)}(t_1,t_2,t_3) =\displaystyle \int_0^{t_1}\int_0^{t_2}   x [[f_1\, f_2], f_3]^{(t_1,t_3,s_1, s_2)}ds_1ds_2.
\eeq
To prove \eqref{exvf-3}, let us shorten notation by setting  $\Psi(t_1, t_2) = e^{t_1f_1}e^{t_2f_2}e^{-t_1f_1}e^{-t_2f_2}$. One has
\begin{align}\label{V-m-3}
\begin{aligned}
 &xV_{[[X_1,X_2], X_3] }^{(f_1,f_2, f_3)}(t_1,t_2,t_3)\\
 &= \dfrac{\partial}{\partial \tau}{{\bigg|_{\tau=t_3}}}\left(
 xe^{t_3f_3}\Psi(t_1, t_2) e^{-t_3f_3} \Psi(t_1, t_2)^{-1}   \Psi(t_1, t_2) e^{ \tau f_3 } \Psi(t_1, t_2)^{-1} e^{ - \tau f_3 }
 \right)  \\
 &= \dfrac{\partial}{\partial \tau}{{\bigg|_{\tau=t_3}}}\left(
 x e^{t_3f_3}\Psi(t_1, t_2)  e^{ (\tau -t_3) f_3 } \Psi(t_1, t_2)^{-1} e^{ - \tau f_3 }
 \right) \\
 &= x e^{t_3f_3}\Psi(t_1, t_2)  f_3 \Psi(t_1, t_2)^{-1} e^{ - t_3 f_3 } -
 xe^{t_3f_3}\Psi(t_1, t_2)   \Psi(t_1, t_2)^{-1}f_3 e^{ - t_3 f_3 }    \\
 &=xe^{t_3f_3}\Psi(t_1, t_2)  f_3 \Psi(t_1, t_2)^{-1} e^{ - t_3 f_3 } -
 xe^{t_3f_3}f_3 e^{ - t_3 f_3 }  \\
 &=xe^{t_3f_3}\Psi(t_1, t_2)  f_3 \Psi(t_1, t_2)^{-1} e^{ - t_3 f_3 } -
 xe^{t_3f_3}\Psi(t_1, 0) f_3 \Psi(t_1,0)^{-1} e^{ - t_3 f_3 }\\
 &= \int_0^{t_2} \frac{\partial}{\partial \sigma }\left( x e^{t_3f_3}\Psi(t_1, \sigma)  f_3 \Psi( t_1, \sigma)^{-1} e^{ - t_3 f_3 }\right) d\sigma \\ &= \int_0^{t_2} x e^{t_3f_3} \frac{\partial}{\partial \sigma }\left( \Psi(t_1, \sigma)  f_3 \Psi( t_1, \sigma)^{-1} \right)e^{ - t_3 f_3 } d\sigma{{.}}
\end{aligned}
\end{align}
To compute the last integral let us begin by observing that,   in view of Definition~\ref{def-V}, one has
\bel{der-1}
\frac{\partial}{\partial \sigma }\Psi(t_1, \sigma)  = \Psi(t_1, \sigma)
V_{[X_1, X_2]}^{(f_1, f_2)}(t_1, \sigma){{.}}  %\Psi( t_1, \sigma)
\eeq
Furthermore,  let us  compute  the derivative $ \frac{\partial}{\partial \sigma }\left( \Psi( t_1, \sigma)^{-1} \right)$,  by differentiating the  relation
\[
 \Psi(t_1, \sigma) \Psi(t_1, \sigma)^{-1} = Id_M
\]
with respect to $\sigma$. We obtain
\begin{align*}
0 &=   \frac{\partial}{\partial \sigma }\left(\Psi(t_1, \sigma) \Psi( t_1, \sigma)^{-1} \right)\\
&=  \frac{\partial}{\partial \sigma }\left(\Psi(t_1, \sigma) \right) \Psi( t_1, \sigma)^{-1} +  \Psi(t_1, \sigma)\frac{\partial}{\partial \sigma }\left( \Psi( t_1, \sigma)^{-1} \right) \\
&= \Psi(t_1, \sigma) V_{[X_1, X_2]}^{(f_1, f_2)}(t_1, \sigma)  \Psi( t_1, \sigma)^{-1}  + \Psi(t_1, \sigma)\frac{\partial}{\partial \sigma }\left( \Psi( t_1, \sigma)^{-1} \right),
\end{align*}
from which we get
\bel{der-2}
\frac{\partial}{\partial \sigma }\left( \Psi( t_1, \sigma)^{-1} \right) = - V_{[X_1, X_2]}^{(f_1, f_2)}(t_1, \sigma)  \Psi( t_1, \sigma)^{-1} .
\eeq
Using \eqref{der-1}, \eqref{der-2}, we can continue the row of equalities in \eqref{V-m-3}, so obtaining
%\begin{align*}
%xV_{[[X_1,X_2], X_3] }^{(f_1,f_2, f_3)}(t_1,t_2,t_3) = \\ \int_0^{t_2} \Big( x e^{t_3f_3}  \frac{\partial}{\partial \sigma }\left( \Psi(t_1, \sigma)  \right) f_3 \Psi( t_1, \sigma)^{-1} e^{ - t_3 f_3 }  & + x e^{t_3f_3}  \Psi(t_1, \sigma)  f_3   \frac{\partial}{\partial \sigma }\left( \Psi( t_1, \sigma)^{-1} \right)e^{ - t_3 f_3 }  \Big)d\sigma =\\
% \int_0^{t_2} \Big( x e^{t_3f_3}  \Psi(t_1, \sigma) V_{[X_1, X_2]}^{(f_1, f_2)}(t_1, \sigma) f_3 \Psi( t_1, \sigma)^{-1} e^{ - t_3 f_3 } &
% \\ - ~x e^{t_3f_3}  \Psi(t_1, \sigma)  f_3 &V_{[X_1, X_2]}^{(f_1, f_2)}(t_1, \sigma)  \Psi( t_1, \sigma)^{-1}e^{ - t_3 f_3 }  \Big)d\sigma  \\
%= \int_0^{t_2} x e^{t_3f_3}   \Psi(t_1, \sigma) \left[ V_{[X_1, X_2]}^{(f_1, f_2)}(t_1, \sigma)   \,,  f_3  \right] &\Psi( t_1, \sigma)^{-1} e^{ - t_3 f_3 }d\sigma.
%\end{align*} 
\begin{align*}
&xV_{[[X_1,X_2], X_3] }^{(f_1,f_2, f_3)}(t_1,t_2,t_3)  \\
&= \int_0^{t_2} \Big( x e^{t_3f_3}  \frac{\partial}{\partial \sigma }\left( \Psi(t_1, \sigma)  \right) f_3 \Psi( t_1, \sigma)^{-1} e^{ - t_3 f_3 } \\
&\quad  + x e^{t_3f_3}  \Psi(t_1, \sigma)  f_3   \frac{\partial}{\partial \sigma }\left( \Psi( t_1, \sigma)^{-1} \right)e^{ - t_3 f_3 }  \Big)d\sigma \\ 
&= \int_0^{t_2} \Big( x e^{t_3f_3}  \Psi(t_1, \sigma) V_{[X_1, X_2]}^{(f_1, f_2)}(t_1, \sigma) f_3 \Psi( t_1, \sigma)^{-1} e^{ - t_3 f_3 }\\
&\quad-x e^{t_3f_3}  \Psi(t_1, \sigma)  f_3 V_{[X_1, X_2]}^{(f_1, f_2)}(t_1, \sigma)  \Psi( t_1, \sigma)^{-1}e^{ - t_3 f_3 }  \Big)d\sigma\\
&= \int_0^{t_2} x e^{t_3f_3}   \Psi(t_1, \sigma) \left[ V_{[X_1, X_2]}^{(f_1, f_2)}(t_1, \sigma)   \,,  f_3  \right] \Psi( t_1, \sigma)^{-1} e^{ - t_3 f_3 }d\sigma.
\end{align*} 
Then, using \eqref{exvf}, we get
\begin{align*}
&xV_{[[X_1,X_2], X_3] }^{(f_1,f_2, f_3)}(t_1,t_2,t_3)\\
&= \int_0^{t_2} x e^{t_3f_3}   \Psi(t_1, \sigma) \left[
\int_0^{t_1}  [f_1\, f_2]^{( \sigma,s_1)}ds_1
  \,,  f_3  \right] \Psi( t_1, \sigma)^{-1} e^{ - t_3 f_3 }d\sigma \\
  & = \int_0^{t_1} \int_0^{t_2}  x e^{t_3f_3}   \Psi(t_1, s_2) \left[
  [f_1\, f_2]^{(s_2,s_1)}
  \,,  f_3  \right] \Psi( t_1, s_2)^{-1} e^{ - t_3 f_3 }ds_1ds_2
\end{align*}
having set $\sigma =s_2$.
Taking into account~\eqref{esempio2}, this is precisely \eqref{exvf-3}.

\end{itemize}

\subsection{Proof of Theorem~\ref{vectorth}. }
%$$\,$$

Theorem~\ref{vectorth} will be proved  as a consequence of the following  result, which establishes  a recursive structure for the vector fields  $V_B^{{\bf f}}({\bf t})$.

\begin{proposition} If $B$ is a canonical bracket  of degree $m$  and   $B=[B_1,B_2^{(m_1)}]$ is its  canonical factorization,
 then, for all $x\in M$,
\bel{formula1}
xV_B^{{\bf f}}({\bf t})=\int_0^{t_{m_1}} x{{{Ad}}}_{\Psi_{B_2}^{{\bf f}_{(2)}}({\bf t}_{(2)})} {{{Ad}}}_{\Psi_{B_1}^{{\bf f}_{(1)}} \left( \underset{m_1}{{\bf t}_{(1)} },\sigma \right) }
\left(\left[V_{B_1}^{{\bf f}_{(1)}}\left(\underset{m_1}{{\bf t}_{(1)}},\sigma\right),
V_{B_2}^{{\bf f}_{(2)}}\left({\bf t}_{(2)}\right)\right]\right)\,d\sigma\,.
\eeq
\end{proposition}
\begin{proof}
{{Consider the map}} $t\mapsto x\Psi_B^{{\bf
f}}(t_1,\ldots,t_{m-1},t)^{-1}$.  Since for all $t$ one has
$$
x\Psi_B^{{\bf f}}(t_1,\ldots,t_{m-1},t) \Psi_B^{{\bf
f}}(t_1,\ldots,t_{m-1},t)^{-1}=x\,,
$$
one gets
\begin{align*}
0&=\frac{\partial}{\partial t}\Big(x\Psi_B^{{\bf f}}(t_1,\ldots,
t_{m-1},t)\Psi_B^{{\bf f}}(t_1,\ldots,t_{m-1},t)^{-1}\Big)\\
&=\frac{\partial}{\partial t}\Big( x\Psi_B^{{\bf
f}}(t_1,\ldots,t_{m-1},t)\Big)
\Psi_B^{{\bf f}}(t_1,\ldots,t_{m-1},t)^{-1}\\
&\quad+x\Psi_B^{{\bf f}}(t_1,\ldots,t_{m-1},t)
\frac{\partial}{\partial t}\Big(
\Psi_B^{{\bf f}}(t_1,\ldots,t_{m-1},t)^{-1}\Big)\\
&= x\Psi_B^{{\bf f}}(t_1,\ldots,t_{m-1},t)V_B^{{\bf
f}}(t_1,\ldots,t_{m-1},
t)\Psi_B^{{\bf f}}(t_1,\ldots,t_{m-1},t)^{-1}\\
&\quad+x\Psi_B^{{\bf f}}(t_1,\ldots,t_{m-1},t)
\frac{\partial}{\partial t}\Big( \Psi_B^{{\bf
f}}(t_1,\ldots,t_{m-1},t)^{-1}\Big)\,.
\end{align*}
If we write $y=x\Psi_B^{{\bf f}}(t_1,\ldots,t_{m-1},t)$, then we
have shown that
$$
yV_B^{{\bf f}}(t_1,\ldots,t_{m-1},
t)\Psi_B^{{\bf f}}(t_1,\ldots,t_{m-1},t)^{-1}\\
+\frac{\partial}{\partial t}\Big( y\Psi_B^{{\bf
f}}(t_1,\ldots,t_{m-1},t)^{-1}\Big)=0\,.
$$
As $x$ varies over $M$, so does $y$, and we can rewrite the above
using the variable $x$ instead of $y$, obtaining
$$
xV_B^{{\bf f}}(t_1,\ldots,t_{m-1},
t)\Psi_B^{{\bf f}}(t_1,\ldots,t_{m-1},t)^{-1}\\
+\frac{\partial}{\partial t}\Big( x\Psi_B^{{\bf
f}}(t_1,\ldots,t_{m-1},t)^{-1}\Big)=0\,,
$$
from which it follows that
\begin{align}
\frac{\partial}{\partial t}\Big(
x\Psi_B^{{\bf f}}(t_1,\ldots,t_{m-1},t)^{-1}\Big) = -xV_B^{{\bf f}}(t_1,\ldots,t_{m-1},t)
\Psi_B^{{\bf f}}(t_1,\ldots,t_{m-1},t)^{-1}\,.\label{bec-c}
\end{align}

{{Setting $m_2\equaldef m-m_1$}} and {{then b}}y \eqref{d-V}, \eqref{bec-c}  we obtain
\begin{align*}
&\frac{\partial}{\partial t} x\Psi_B^{{\bf f}}\left(\underset{m}{\bf t},t\right)\\
&=\frac{\partial}{\partial t}\Big(
x\Psi_{B_1}^{{\bf f}_{(1)}}({\bf t}_{(1)}) \Psi_{B_2}^{{\bf
f}_{(2)}}\left(\underset{{{m_2}}}{{\bf t}_{(2)}},t\right) \Psi_{B_1}^{{\bf f}_{(1)}}({\bf
t}_{(1)})^{-1}
\Psi_{B_2}^{{\bf f}_{(2)}}\left(\underset{{{m_2}}}{{\bf t}_{(2)}},t\right)^{-1}\Big)\\
&=x\Psi_{B_1}^{{\bf f}_{(1)}}({\bf t}_{(1)}) \Psi_{B_2}^{{\bf
f}_{(2)}}\left(\underset{{{m_2}}}{{\bf t}_{(2)}},t\right) V_{B_2}^{{\bf f}_{(2)}}\left(\underset{{{m_2}}}{{\bf t}_{(2)}},t\right) \Psi_{B_1}^{{\bf f}_{(1)}}({\bf t}_{(1)})^{-1}
\Psi_{B_2}^{{\bf f}_{(2)}}\left(\underset{{{m_2}}}{{\bf t}_{(2)}},t\right)^{-1}\\
&\quad-x\Psi_{B_1}^{{\bf f}_{(1)}}({\bf t}_{(1)}) \Psi_{B_2}^{{\bf
f}_{(2)}}\left(\underset{{{m_2}}}{{\bf t}_{(2)}},t\right) \Psi_{B_1}^{{\bf f}_{(1)}}({\bf
t}_{(1)})^{-1} V_{B_2}^{{\bf f}_{(2)}}\left(\underset{{{m_2}}}{{\bf t}_{(2)}},t\right)
\Psi_{B_2}^{{\bf f}_{(2)}}\left(\underset{{{m_2}}}{{\bf t}_{(2)}},t\right)^{-1}\\
&=x\Psi_{B_1}^{{\bf f}_{(1)}}({\bf t}_{(1)}) \Psi_{B_2}^{{\bf
f}_{(2)}}\left(\underset{{{m_2}}}{{\bf t}_{(2)}},t\right) \Psi_{B_1}^{{\bf f}_{(1)}}({\bf
t}_{(1)})^{-1}
\Psi_{B_2}^{{\bf f}_{(2)}}\left(\underset{{{m_2}}}{{\bf t}_{(2)}},t\right)^{-1}\\
&\qquad\Big(\Psi_{B_2}^{{\bf f}_{(2)}}\left(\underset{{{m_2}}}{{\bf t}_{(2)}},t\right)
\Psi_{B_1}^{{\bf f}_{(1)}}({\bf t}_{(1)}) V_{B_2}^{{\bf f}_{(2)}}\left(\underset{{{m_2}}}{{\bf t}_{(2)}},t\right) \Psi_{B_1}^{{\bf f}_{(1)}}({\bf t}_{(1)})^{-1}
\Psi_{B_2}^{{\bf f}_{(2)}}\left(\underset{{{m_2}}}{{\bf t}_{(2)}},t\right)^{-1}\Big)\\
&\quad-x\Psi_{B_1}^{{\bf f}_{(1)}}({\bf t}_{(1)}) \Psi_{B_2}^{{\bf
f}_{(2)}}\left(\underset{{{m_2}}}{{\bf t}_{(2)}},t\right) \Psi_{B_1}^{{\bf f}_{(1)}}({\bf
t}_{(1)})^{-1}
\Psi_{B_2}^{{\bf f}_{(2)}}\left(\underset{{{m_2}}}{{\bf t}_{(2)}},t\right)^{-1}\\
&\qquad\Big(\Psi_{B_2}^{{\bf f}_{(2)}}\left(\underset{{{m_2}}}{{\bf t}_{(2)}},t\right)
V_{B_2}^{{\bf f}_{(2)}}\left(\underset{{{m_2}}}{{\bf t}_{(2)}},t\right)
\Psi_{B_2}^{{\bf f}_{(2)}}\left(\underset{{{m_2}}}{{\bf t}_{(2)}},t\right)^{-1}\Big)\\
&=x\Psi_B^{{\bf f}}\left(\underset{m}{\bf t},t\right) \Big(\Psi_{B_2}^{{\bf
f}_{(2)}}\left(\underset{{{m_2}}}{{\bf t}_{(2)}},t\right) \Psi_{B_1}^{{\bf f}_{(1)}}({\bf
t}_{(1)}) V_{B_2}^{{\bf f}_{(2)}}\left(\underset{{{m_2}}}{{\bf t}_{(2)}},t\right)
\Psi_{B_1}^{{\bf f}_{(1)}}({\bf t}_{(1)})^{-1}
\Psi_{B_2}^{{\bf f}_{(2)}}\left(\underset{{{m_2}}}{{\bf t}_{(2)}},t\right)^{-1}\Big)\\
&\quad-x\Psi_B^{{\bf f}}\left(\underset{m}{\bf t},t\right) \Big(\Psi_{B_2}^{{\bf
f}_{(2)}}\left(\underset{{{m_2}}}{{\bf t}_{(2)}},t\right) V_{B_2}^{{\bf f}_{(2)}}\left(\underset{{{m_2}}}{{\bf t}_{(2)}},t\right) \Psi_{B_2}^{{\bf f}_{(2)}}\left(\underset{{{m_2}}}{{\bf t}_{(2)}},t\right)^{-1}\Big)\,,
\end{align*}
from which it follows that
\begin{align*}
V_B^{{\bf f}}\left(\underset{m}{\bf t},t\right) &=\Psi_{B_2}^{{\bf f}_{(2)}}\left(\underset{{{m_2}}}{{\bf t}_{(2)}},t\right) \Psi_{B_1}^{{\bf f}_{(1)}}({\bf t}_{(1)})
V_{B_2}^{{\bf f}_{(2)}}\left(\underset{{{m_2}}}{{\bf t}_{(2)}},t\right) \Psi_{B_1}^{{\bf
f}_{(1)}}({\bf t}_{(1)})^{-1}
\Psi_{B_2}^{{\bf f}_{(2)}}\left(\underset{{{m_2}}}{{\bf t}_{(2)}},t\right)^{-1}\\
&\quad-\Psi_{B_2}^{{\bf f}_{(2)}}\left(\underset{{{m_2}}}{{\bf t}_{(2)}},t\right)
V_{B_2}^{{\bf f}_{(2)}}\left(\underset{{{m_2}}}{{\bf t}_{(2)}},t\right) \Psi_{B_2}^{{\bf
f}_{(2)}}\left(\underset{{{m_2}}}{{\bf t}_{(2)}},t\right)^{-1}\,,
\end{align*}
so that
\bel{formula 1}
V_B^{{\bf f}}\left(\underset{m}{\bf t},t\right) ={{{Ad}}}_{\Psi_{B_2}^{{\bf
f}_{(2)}}\left(\underset{{{m_2}}}{{\bf t}_{(2)}},t\right)} \Biggl({{{Ad}}}_{\Psi_{B_1}^{{\bf
f}_{(1)}}({\bf t}_{(1)})}\Big( V_{B_2}^{{\bf f}_{(2)}}\left(\underset{{{m_2}}}{{\bf t}_{(2)}},t\right)\Big) -V_{B_2}^{{\bf f}_{(2)}}\left(\underset{{{m_2}}}{{\bf t}_{(2)}},t\right)\Biggr){{.}}\,
\eeq
For  any
$y\in M$,one has 
\begin{align*}
&\frac{\partial}{\partial\sigma} \Big(y\Psi_{B_1}^{{\bf f}_{(1)}}\left(\underset{m_1}{{\bf t}_{(1)}},\sigma\right) V_{B_2}^{{\bf f}_{(2)}}\left(\underset{{{m_2}}}{{\bf t}_{(2)}},t\right)
\Psi_{B_1}^{{\bf f}_{(1)}}\left(\underset{m_1}{{\bf t}_{(1)}},\sigma\right)^{-1}\Big)\\
&=y\Psi_{B_1}^{{\bf f}_{(1)}}\left(\underset{m_1}{{\bf t}_{(1)}},\sigma\right) V_{B_1}^{{\bf
f}_{(1)}}\left(\underset{m_1}{{\bf t}_{(1)}},\sigma\right) V_{B_2}^{{\bf f}_{(2)}}\left(\underset{{{m_2}}}{{\bf t}_{(2)}},t\right)
\Psi_{B_1}^{{\bf f}_{(1)}}\left(\underset{m_1}{{\bf t}_{(1)}},\sigma\right)^{-1}\\
&\quad-y\Psi_{B_1}^{{\bf f}_{(1)}}\left(\underset{m_1}{{\bf t}_{(1)}},\sigma\right) V_{B_2}^{{\bf f}_{(2)}}\left(\underset{{{m_2}}}{{\bf t}_{(2)}},t\right)
V_{B_1}^{{\bf f}_{(1)}}\left(\underset{m_1}{{\bf t}_{(1)}},\sigma\right)
\Psi_{B_1}^{{\bf f}_{(1)}}\left(\underset{m_1}{{\bf t}_{(1)}},\sigma\right)^{-1}\\
&=y\Psi_{B_1}^{{\bf f}_{(1)}}\left(\underset{m_1}{{\bf t}_{(1)}},\sigma\right)
\left[V_{B_1}^{{\bf f}_{(1)}}\left(\underset{m_1}{{\bf t}_{(1)}},\sigma\right), V_{B_2}^{{\bf
f}_{(2)}}\left(\underset{{{m_2}}}{{\bf t}_{(2)}},t\right)\right] \Psi_{B_1}^{{\bf f}_{(1)}}\left(\underset{m_1}{{\bf t}_{(1)}},\sigma\right)^{-1}{{.}}\,
\end{align*}
If we take $y=x\Psi_{B_2}^{{\bf f}_{(2)}}\left(\underset{{{m_2}}}{{\bf t}_{(2)}},t\right)$,
$x\in M$, and use the fact that
$$
\Psi_{B_1}^{{\bf f}_{(1)}}\left(\underset{m_1}{{\bf t}_{(1)}},\sigma\right) V_{B_2}^{{\bf
f}_{(2)}}\left(\underset{{{m_2}}}{{\bf t}_{(2)}},t\right) \Psi_{B_1}^{{\bf f}_{(1)}}\left(\underset{m_1}{{\bf t}_{(1)}},\sigma\right)^{-1}= V_{B_2}^{{\bf f}_{(2)}}\left(\underset{{{m_2}}}{{\bf t}_{(2)}},t\right)\qquad {\rm when}\quad \sigma=0\,,
$$
by \eqref{formula 1} we get \eqref{formula1}, if  $t=t_m$.  This concludes the proof. 
\end{proof}
%$$
%xV_B^{{\bf f}}({\bf t}^{\tau})=\int_0^{t_{m_1}} x{\rm
%Ad}_{\Psi_{B_2}^{{\bf f}_{(2)}}({\bf t}_{(2)}^{\tau})} {\rm
%Ad}_{\Psi_{B_1}^{{\bf f}_{(1)}}({\bf t}_{(1)}^\sigma)}
%\Big(\Big[V_{B_1}^{{\bf f}_{(1)}}(({\bf t}_{(1)}^{m_1},\sigma)),
%V_{B_2}^{{\bf f}_{(2)}}({\bf t}_{(2)}^{\tau})\Big]\Big)\,d\sigma\,,
%$$
%or, equivalently, \eqref{formula1}. This concludes the proof.\end{proof}

{\sc Proof of  Theorem~\ref{vectorth}.}
%\begin{proofof}
Let us proceed by induction on $m=deg(B)$. For $m=1$, the thesis simply follows from \eqref{exvf1}. For $m>1$
let us assume the inductive hypothesis for  each of the two subbrackets $B_1$, $B_2$ appearing in the canonical factorization $[B_1, B_2^{(m_1) }]$  of $B$. So, 
 \bel{sub1}
  xV_{B_1}^{{\bf f}_{(1) } } \left(\underset{m_1}{{\bf t}_{(1)}},{ s_{m_1}} \right)=  \int_0^{t_{1}}\cdots\int_0^{ t_{m_1-1}} x{ B_1}^{\left(\underset{\{m_{11}, m_1\}}{{\bf t}_{(1)}},{ s_{m_1}},{\bf s}_{(1)}  \right)}({\bf f}_{(1) }) ds_1\dots ds_{m_1-1}\,\,,\eeq
\bel{sub2}
  xV_{B_2}^{ {\bf f}_{(2) }} \left({\bf t}_{(2)} \right) =  \int_0^{t_{m_1+1}}\cdots\int_0^{t_{m-1}} x { B_2}^{\left(\underset{m_{21}}{{{{\bf t}_{(2)}}}},{\bf s}_{(2)} \right)}({\bf f}_{(2) })ds_{m_1+1}\dots ds_{m-1}\, .\eeq
   If $m_1 = 1$ ({{respectively,}} if  $m_1=m-1$, i.e., $m_2= m-m_1=1$ ), we mean that formula \eqref{sub1}  ({{respectively,}} \eqref{sub2}) reads $xV_{B_1}^{{\bf f}_{(1) } } {{\left(\underset{m_1}{{\bf t}_{(1)}},{ s_{m_1}} \right)}} = xV_{X_1}^{f_1} (s_1) =x f_1$  ({{respectively,}} $xV_{B_2}^{ {\bf f}_{(2) }} ({\bf t}_{(2)} ) =
 x V_{X_1}^{f_m} (t_m)  =x f_m$).

By applying  \eqref{formula1} we obtain
\begin{align*}
xV_B^{{\bf f}}({\bf t}) &= \displaystyle\int_0^{t_{1}} \cdots \int_0^{t_{m-1}}  x{{{Ad}}}_{\Psi_{B_2}^{{\bf f}_{(2)}}({\bf t}_{(2)})} {{{Ad}}}_{\Psi_{B_1}^{{\bf f}_{(1)}} \left(\underset{m_1}{{\bf t}_{(1)}},{ s_{m_1}} \right)}\nonumber\\
&\qquad\qquad\qquad\quad\left[{ B_1}^{\left(\underset{\{m_{11}, m_1\}}{{\bf t}_{(1)}},{ s_{m_1}},{\bf s}_{(1)}  \right)}( {\bf f}_{(1) })\,,
 { B_2}^{\left(\underset{m_{21}}{{{{\bf t}_{(2)}}}},{\bf s}_{(2)} \right)}({\bf f}_{(2) })
 \right]ds_1 \dots ds_{m-1}.
\end{align*}
Since, by Definition ~\ref{def-expbra},
\begin{align*}%\label{afbtv} %auxiliary formula for a bracket twisted value
 x{{{Ad}}}_{\Psi_{B_2}^{{\bf f}_{(2)}}({\bf t}_{(2)})} {{{Ad}}}_{\Psi_{B_1}^{{\bf f}_{(1)}} \left(\underset{m_1}{{\bf t}_{(1)}},{ s_{m_1}} \right)}
\left[{ B_1}^{\left(\underset{\{m_{11}, m_1\}}{{\bf t}_{(1)}},{ s_{m_1}},{\bf s}_{(1)}  \right)}( {\bf f}_{(1) })\,,
 { B_2}^{\left(\underset{m_{21}}{{{{\bf t}_{(2)}}}},{\bf s}_{(2)} \right)}({\bf f}_{(2) })
 \right] = xB^{ \left(\underset{m_1}{\bf t},{\bf s} \right) } ({\bf f}),
\end{align*}
 the proof is concluded.
\red{\proofend}
%\end{proofof}

%We shall prove directly the assertions concerning $m=2,3$. Afterwards  we shall prove the general statement by induction on the degree $m$.

\section{{{``}}$C^B$'' regularity }\label{concluding-reg}$\,$

 The  main results of this paper  remain valid  also when vector fields $f_i$ fail to be $C^{\infty}$, provided suitable $C^r$  hypotheses are assumed. To state them, given an $m$-tuple of vector fields ${\bf f} = (f_1, \ldots, f_{m})$ on $M$ and a canonical bracket $B$ of degree $m$, we  shall define the notion of ${\bf f}$ {\it of  class $C^B$}. %  (or ${\bf f} \in C^B$)
 Roughly speaking, it  means that all components $f_i${{,}} $i=1, \ldots, m$, possess  the minimal order of differentiation for which  $B( {\bf f})$ {{c}}an be computed everywhere an{{d}} is continuous.

As a byproduct of the integral representation provided in Theorem \ref{integralth} we get versions of the asymptotic formulas (and of   Chow-Rashevski's controllability theorem)  under quite low regularity hypotheses.
\if
valid under  classical results such as the asymptotic formulas and (a their corollary by means of the open mapping theorem) the  under minimal regularity assumptions for standard proofs to remain valid. Of course, the asymptotic formulas are also an immediate  consequence of integral formulas in Theorem~\ref{integralth}, but, as explained in the Introduction, a ``true'' application of these formulas will be presented in a forthcoming paper, where the regularity assumptions for vector fields $f_i$ will be relaxed considerably by requiring that their highest order derivatives needed to compute $B({\bf f})$ are only defined almost everywhere, locally bounded and measurable.

\fi

\subsection{Number of differentiations}
\medbreak\noindent To give a precise meaning to the expression {\it {{``all components $f_i${{,}} $i=1, \ldots, m$,}} possess the minimal order of differentiation for  which  $B( {\bf f})$ {{c}}an be computed everywhere an{{d}} is continuous"  } we need some formalism concerning the way    any bracket $B$ can be regarded as constructed in a
recursive way by iterated bracketings. In this recursive
construction, each subbracket $S$ undergoes a certain number of
bracketings until $B$ is obtained. When we plug in vector fields $f_j$
for the indeterminates $X_j$, each bracketing involves a
differentiation. So we will refer to this ``number of bracketings'' as
``the number of differentiations of $S$ in $B$,'' and use the
expression $\Delta(S;B)$ to denote it.
Naturally, this will only make sense for brackets $B$ and subbrackets
$S$ such that $S$ only occurs once as a substring of $B$. For more
general brackets, one must define a ``subbracket of $B$'' to be not
just a string that occurs as a substring of $B$ and is a bracket, but
as an {\em occurrence} of such a string, so that, {{e.g.}}, the two
occurrences of $X_1$ in $B=[X_1,[X_1,X_2]]$ count as different
subbrackets. Notice that the number of differentiations of $X_1$ in
$B$ is $1$ for the first one and $2$ for the second one. In order to
avoid this extra complication, we will confine ourselves to
semicanonical brackets, for which this problem does not arise, because
{ a subbracket $S$ of a semicanonical bracket $B$ can only occur
once as a substring of $B$.}

The precise definition of $\Delta(S;B)$,  $B$ being canonical,
and $S\in Subb(B)$, is by a backwards recursion on $S$:
\begin{itemize}
\item[{{(i)}}] $\Delta(B;B)\,\,\equaldef\,\,0$; \item[{{(ii)}}]
$\Delta(S_1;B)\,\,\equaldef\,\,\Delta(S_2;B)\,\,\equaldef\,\,
1+\Delta([S_1,S_2];B)\,$.
\end{itemize}
It is then easy to prove by induction that
$$
\Delta(S;B)=nrbr(S;B)-nlbr(S;B)\,,
$$
where $nrbr(S;B)$ {{are}} the number of right brackets that occur in $B$ to
the right of $S$, and $nlbr(S;B)$ {{are}} the number of left brackets that
occur in $B$ to the right of $S$.
\if
(This follows from the fact that
every time a bracket $[S_1,S_2]$ is created from brackets $S_1,S_2$,
and $S$ is a subbracket of $S_i${{,}} $i=1$ or $i=2$, then the identity
$$
nrbr(S;[S_1,S_2])-nlbr(S;[S_1,S_2]) =1+nrbr(S,S_i)-nlbr(S,S_i)
$$
holds.)
\fi
 For example, if
$ B=[X_3,[[[[X_4,X_5],$\\$X_6],X_7],[X_8,[X_9,X_{10}]]]] $
 then $\Delta([X_4,X_5];B)=4$.

\if
Hence $\Delta(X_j;B)$ is the number of right brackets
that occur in $B$ to the right of $X_j$, minus the number of left
brackets that occur in $B$ to the right of $X_j$. For example,
\begin{align*}
\Delta_1(X_1)&=0\,,\\
\Delta_1([X_1,X_2])=\Delta_1([X_1,X_2])&=1\,,\\
\Delta_1([[X_1,X_2],X_3])&=2\,,\\
\Delta_1([X_1,[X_2,X_3]])&=1\,,\\
\Delta_2([X_1,[X_2,X_3]])=\Delta_3([X_1,[X_2,X_3]])&=2\,.
\end{align*}
\fi
It is easy to see that, if $(B_1,B_2)$ is the canonical factorization
of $B$, and \mbox{$deg(B_1)=m_1$}, $deg(B_2)=m_2$, then
$$
\Delta(X_j;B)=\left\{\begin{array}{rll}
\Delta(X_j;B_1)+1&\quad{\rm if}\quad&j\in\{1,\ldots,m_1\}\\
\Delta(X_{j-m_1};B_2)+1&\quad{\rm
if}\quad&j\in\{m_1+1,\ldots,m_1+m_2\}\,.
\end{array}\right.
$$
The following trivial but important identity then holds: {{i}}f  $X_j $ is a subbracket of $S$ and $S$ is a subbracket $B$,
$$
\Delta_j(B)=\Delta_j(S)+\Delta(S;B)\quad {\rm if}\quad X_j \quad {\rm and}\,\,\,S\quad {\rm are\,\, subbrackets \,\,of}\,\, B.
$$

%\if Our definition
%will involve the notions of ``vector field of class $C^{k}$'' and
%``vector field of class $C^{k-1;L}$'' for general nonnegative integers
%$k$. The meaning of this concept is clear in all cases, with the only
%exception of that of a ``vector field of class $C^{-1;L}$.'' So we
%first fill this gap by providing the missing definition.

%\begin{definition}\label{defsval}
%A {\em vector field of class $C^{-1;L}$} is an upper semicontinuous
%set-valued vector field with compact convex nonempty values.\proofend
%\end{definition}

%\fi

\begin{definition}[Class $C^{B+k}$]\label{asindef}Let $m,\mu,\nu, k$ be nonnegative integers  such that
$\nu\ge m+\mu$, $m\ge 1$.
Given
 a semicanonical bracket $B$ of degree $m$ such that
$B=B_0^{(\mu)}$, $B_0$ being canonical, and a $\nu$-tuple ${\bf
f}=(f_1,\ldots,f_{\nu})$ of vector fields, 
we say that
 ${\bf f}$ {\em is of class $C^{B+k}$} if $f_j$ is of class
$C^{\Delta_j(B)+k}$ for each $j\in\{1\!+\!\mu,\ldots,m\!+\!\mu\}.$
 We also write ${\bf f}\in C^{B+k}$  to indicate that ${\bf f}$ is of class
$C^{B+k}$.  Finally, we simplify
the notation by just writing $C^{B}$  instead of $C^{B+0}$.
\end{definition}

\begin{remark} The above definition can be adapted in a obvious way to the case when $M$ is just a manifold of class $C^\ell$
for $\ell \ge 1+k+\max\Big\{\Delta_j(B):j\in\{1\!+\!\mu,\ldots,m\!+\!\mu\}\Big\}$.  
\end{remark}

For example, suppose that
$
B=[[X_1,X_2],[[X_3,X_4],X_5]]$ and $ {\bf
f}=(f_1,f_2,f_3,f_4,f_5)$.
Then
${\bf f}\in C^{B}$ if and only if $f_1,f_2,f_5\in C^{2}$ and $f_3,f_4\in C^{3}. $ It is then easy to verify the following result:

\if
For a second example, suppose that
$$
B=[[X_3,X_4],X_5]\qquad{\rm and}\qquad{\bf
f}=(f_1,f_2,f_3,f_4,f_5,f_6,f_7)\,.
$$
Then
\begin{align*}
{\bf f}\in C^{B\phantom{;L}}&\quad{\rm iff}\qquad \quad f_3,f_4\in
C^{2\phantom{;L}}\,\,\,
\,\,\,{\rm and}\,\,\,\,\,f_5\in C^{1}\,,\\
{\bf f}\in C^{B;L}&\quad{\rm iff}\qquad\quad f_3,f_4\in C^{2;L}
\,\,\,\,\,\,{\rm and}\,\,\,\,\,f_5\in C^{1;L}\,.
\end{align*}
\fi

\begin{proposition}
Assume that we are given data $B$, $m$, $B_0$, $\mu$, $k$, $\nu$, and an $\nu$-tuple ${\bf f}=(f_1,\ldots,f_\nu)$ as in Definition
\ref{asindef}. Let $(B_1,B_2)$ be the factorization of $B$. Then
 ${\bf f}\in C^{B+k}$ if and only if 
${\bf f}\in C^{B_1+k+1}$ and ${\bf
f}\in C^{B_2+k+1}$.
\end{proposition}
It then follows, by an easy induction on the subbrackets $S$ of $B$,
that one can define $S({\bf f})$ for every  subbracket $S$ of $B$ as a %proper
true vector field, by simply letting
$$
S({\bf f})=[S_1({\bf f}),S_2({\bf f})]\qquad{\rm if}\quad
S=[S_1,S_2]\,.
$$
The resulting vector field $S({\bf f})$ is of class
$ C^{\Delta(S;B)+k }$ as soon as
${\bf f}$ is of class $C^{B+k}$. In particular:

\begin{itemize}
\item[{{(i)}}] If ${\bf f}\in C^{B+k}$, then $B({\bf f})$ is a vector field
on $M$ of class $C^{k}$. %\item{} %If $k\ge 1$ and ${\bf f}$ belongs to
%$C^{B+k-1;L}$, then $B({\bf f})$ is a vector field on $M$ of class
%$C^{k-1;L}$.
%\end{itemize}
%\begin{itemize}
\item[{{(ii)}}] {{I}}f ${\bf f}\in C^{B}$ then $B({\bf f})$ is a
 continuous vector field.
%\item{} if ${\bf f}\in C^{B;L}$ then $B({\bf f})$ is a locally
%Lipschitz vector field.\proofend
\end{itemize}

\subsection{Representation with low regularity, asymptotic formulas, and  Chow-Rashevski's theorem}
One can easily verify that  $\left(x,\underset{{{\{m_1,m\}}}}{\bf t}{{, s_m}}, {\bf s} \right) \mapsto B({\bf f})^{\left(\underset{{{\{m_1,m\}}}}{\bf t}{{, s_m}}, {\bf s} \right)}$ is well-defined and continuous for any canonical bracket $B$ with canonical factorization $B= [B_1, B_2^{(m_1)}]$, where $1\le m_1 < deg(B)$, and any  ${\bf f} \in C^B$. Moreover, with obvious reinterpretation of the notation  one easily obtains the following low regularity version of  Theorem~\ref{integralth}:
\begin{theorem}[Integral representation with $C^B$ regularity]\label{integralth-low} Let $B$ be a canonical iterated bracket of $deg(B)=m\ge 1$, and let  ${\bf f}$ be an $m$-tuple of vector fields on  $M$ of class $C^B$. Then, for every $m-tuple $ ${\bf t}=(t_1,\dots,t_m)\in \rr^m$ one has
%\bel{intform}
%\begin{multlined}
\begin{align*}
x\Psi_{ B}^{\bf f} ({\bf t}) = x + \int_0^{t_1}\cdots\int_0^{t_m} x\Psi_{ B}^{\bf f}\left(\underset{m}{\bf t},s_m\right) { B}({\bf f})^{\left(\underset{{{\{m_1,m\}}}}{\bf t}{{, s_m}},{\bf{s}}\right)}ds_1\dots ds_m{{.}}\,\,\,
%\end{multlined}
\end{align*}
\end{theorem}

As an  almost obvious  byproduct we get the following asymptotic {{formula}} under low regularity assumptions.
 \if
 \footnote{We do not mean, of course, that such asymptotic formulas and the consequent  controllability theorem  cannot be established without the help of integral representations.}
 \fi
%A first and major application consists in the derivation of asymptotic formulas.

%As usual we fix a differential manifold $M$, (which, when necessary, is assumed to by some open subset of $\rr^n$, $n\ge 1$)  $B$ a canonical bracket of $deg(B)=m \ge 1$,
%and ${\bf f}=(f_1, \ldots, f_m)$ an $m$-tuple of vector fields on $M$.

\begin{theorem}[Asymptotic {{formula}}]\label{asyformcor}
%Let $\Omega$ be an open subset of $\rr^n$,  $B\in ITB_{can}({\bf X})$, and ${\bf f} \in VF^B(\Omega)$ such that each component of ${\bf f}$ is at least of class $C^1$. Then for all $x\in \Omega$
Let $B$ be a canonical iterated bracket of $deg(B)=m\ge 1$, and ${\bf f}$ an $m$-tuple of vector fields on $M$. Assume that ${\bf f}\in C^B$. Then
we have
\begin{align}\label{asyform}
x\Psi_B^{{\bf f}}(t_1, \dots, t_m) = x + t_1 \cdots t_m  B({\bf f})(x) + o(t_1 \cdots t_m )\nonumber
\end{align}
as ${{\|}}(t_1, \ldots,  t_m){{\|}} \to 0$. %, where $m = deg(B)$.
\end{theorem}

 In turn,  as a consequence of the asymptotic formulas above (and via a standard application of the  open mapping theorem, see{{,}} e.g.{{,}} \cite{RaSu01}),  one gets a low {{r}}egularity version of    Chow-Rashevski's controllability theorem:

\begin{theorem}[Chow-Rashevski]\label{ChowIntro}
Let $\{f_1,\dots\red{,}f_r\}$ be a family of ($C^1$) %continuous %$C^1$
vector fields on  $M$. Let us consider the driftless control system
\bel{sistema}
\dot x = \sum_{i=1}^r u_if_i(x)
\eeq
with control constraints $|u_i|\leq 1${{,}} $i=1, \dots,r$. Let $x_*\in M$, and let
$
 B_1,\dots,B_\ell$ and
$ {\bf f}_1, \dots, {\bf f}_{\ell}$
be canonical iterated brackets, and finite collections of the vector fields $f_i$, $i=1,\dots r$, respectively,
such that:
\begin{itemize}
\item[{{(i)}}]   {{F}}or every $j=1,\dots,\ell$, ${\bf f}_j \in C^{B_j}${{;}} %VF^{B_j}(M)
\item[{{(ii)}}] ${{\textnormal{span} \Big\{B_1( {\bf f}_1 )(x_* ), \dots,  B_\ell( {\bf f}_\ell )(x_* )   \Big\} = T_{x_*} M.}}$
%\begin{equation}\label{rancond}
%span \Big\{B_1( {\bf f}_1 )(x_* ), \dots,  B_\ell( {\bf f}_\ell )(x_* )   \Big\} = T_{x_*} M
%\end{equation}
\end{itemize}
Then the control system $\eqref{sistema}$ is locally controllable from $x_*$ in small time.
More precisely, if $d$ is  {{the}} Riemannian distance defined on an open set $A$ containing the  point $x_*$, and if $k$ is the maximum of the degrees of the iterated Lie brackets $B_j$, then there exist a neighborhood $U\subset A$ of $x_*$ and a positive constant $C$ such that for every $x\in U$ one has
\bel{min-time}
T(x)\leq C {{d}}(x,x_*{{)}}^{\frac{1}{k}},\nonumber
\eeq
where $T(x)$ denotes the minimum time to reach $x$ over the set of admissible controls, provided that this set contains the piecewise constant %bang-bang
 controls $t {{\mapsto}} (u_1(t), \dots, u_m(t) )$ such that at each time $t$ only one of the numbers $u_i(t)$, $i=1, \dots, m$, is nonzero.
\end{theorem}

\begin{remark}\label{c0rem}  In view of some  arguments utilized in \cite{RaSu01},  the $C^1$-regularity assumption  for  the vector fields $f_i$  in Theorem \ref{ChowIntro} may be further weakened: in fact,   the only needed regularity hypotheses  are those stated {{in}} point ${{\textnormal{(i)}}}$. The latter, in turn, allow for some of the vector fields $f_i$ to be just continuous, so that the corresponding   flows are set-valued maps. We refer to Subsection \ref{Subsec-nonsmooth} for other considerations on the regularity question.

\end{remark}

 \section{Concluding remarks}\label{concludingsec}

\subsection{On the {{``}}adjoint'' structure of integrating brackets} The crucial difference between integrating brackets of degree $2$ and integrating brackets of degree greater than $2$ consists in the fact that while    the former are  {\it adjoint} to the corresponding Lie brackets, namely
$
 [f_1, f_2]^{({{s_2}},s_1)}  \equaldef\displaystyle Ad_{e^{{{s_2}} f_2}e^{s_1f_1}} \left[f_1\,,\,f_2 \right],
$
the latter in general include intermediate adjoining operations. Indeed this  is  already true  when the degree is equal to three, in that  two adjoinings are needed. For instance:
$$
[[f_1,f_2],f_3]^{(t_1,{{s_3}},s_1,s_2)}
=
Ad_{e^{{{s_3}}f_3}e^{t_1 f_1}e^{s_2f_2}e^{-t_1f_1}e^{-s_2 f_2}}
 \Big[ Ad_{e^{s_2 f_2}e^{s_1f_1}} \left[f_1\,,\,f_2 \right]  \,,\, f_3 \Big]{{.}}
%e^{s_2 f_2}e^{t_1f_1}e^{-s_2f_2}e^{-t_1 f_1}e^{-t_3f_3}
$$
In fact, this would not obstruct the possibility that
\bel{adj2}
[[f_1,f_2],f_3]^{(t_1,{{s_3}},s_1,s_2)}
=
Ad_{\phi(t_1,{{s_3}}, s_1,s_2)}
 \Big[  \left[f_1\,,\,f_2 \right]  \,,\, f_3 \Big]  %Ad_{e^{t_2 f_2}e^{s_1f_1}}
%e^{s_2 f_2}e^{t_1f_1}e^{-s_2f_2}e^{-t_1 f_1}e^{-t_3f_3}
\eeq
for some $(t_1,{{s_3}}, s_1,s_2)$-dependent diffeomorphism $\phi$. However, let us point out that if $\eqref{adj2}$ were standing, then, in view of the integral  representation provided by Theorem~\ref{integralth}, the vanishing of the iterated Lie bracket $[[f_1,f_2],f_3]$ would imply that $\Psi_{[[X_1,X_2],X_3]}^{(f_1,f_2,f_3)} = Id_M$.    Notice incidentally that by Proposition~\ref{ad-br-cor}, \eqref{adj2} holds true for any $f_3$, as soon as $0=[[f_1,f_2],f_2] = [[f_1,f_2],f_1] = 0$. 

Yet, in general  \eqref{adj2} does not hold, so that in general one has  $\Psi_{[[X_1,X_2],X_3]}^{(f_1,f_2,f_3)} \neq Id_M$. This is in fact what we get from the following simple example:

\begin{example}\label{concluding-ex}
%;
 %unless any of the $t_i$, $i=1,2,3$,vanishes:
% we say  that $f_1,f_2,f_3$ {\it do not commute}.
In $M = \rr^2$ let us consider   the linear vector fields
$f_i (x, y)=  A_i \left(\begin{array}{c} x \\ y \end{array}  \right) $, $i=1,2,3$,
where

%\begin{gather*}
$$
A_1 = \left(
\begin{array}{cc}
 0 & 0 \\
 1 & 0 \\
\end{array}
\right),   \quad
 A_2=
 \left(
\begin{array}{cc}
 0 & 1 \\
 0 & 0 \\
\end{array}
\right), \quad
A_3=\left(
\begin{array}{cc}
 1 & 0 \\
 0 & 0 \\
\end{array}
\right) .
$$
%\end{example}

%\end{gather*}

Let us show  that $[[f_1, f_2], f_3] \equiv 0$, and nevertheless $\Psi_{[[X_1, X_2], X_3]}^{(f_1,f_2,f_3) } \neq Id_M$.
Clearly $[f_1,f_2]$, $[f_1,f_2]^{({{s_2}},s_1)}$, $[[f_1, f_2], f_3]$, $[[f_1, f_2], f_3]^{(t_1,{{s_3}},s_1,s_2)}$   are also linear (parameterized) vector fields: the corresponding matrices are defined, respectively, as
\begin{gather*}
[A_1, A_2]\equaldef {{A_1A_2-A_2A_1}}, \\
[A_1, A_2]^{({{s_2}},s_1)} \equaldef {{e^{s_2A_2}  e^{s_1A_1} [A_1, A_2] e^{-s_1A_1} e^{-s_2A_2}}}, \\
[[A_1, A_2], A_3]\equaldef {{A_1A_2A_3-A_2A_1A_3 -A_3A_1A_2+A_3A_2A_1}}{{,}}\\
\begin{small}
[[A_1, A_2], A_3]^{(t_1,{{s_3}},s_1,s_2)} \equaldef {{e^{s_3A_3}e^{t_1 A_1}e^{s_2A_2}e^{-t_1A_1}e^{-s_2 A_2} \left[ [A_1, A_2]^{(s_2,s_1)}, \, A_3 \right]
e^{s_2 A_2} e^{t_1A_1}  e^{-s_2A_2} e^{-t_1 A_1} e^{-s_3A_3}}}{{.}}\end{small}
{{\footnotemark}}
\end{gather*}
\footnotetext{More generally, if $B$ is a canonical bracket of $deg(B)=m \ge 1$ with canonical factorization $B=[B_1, B_2^{(m_1)}]$ for $1\le m_1 <m$, and ${\bf f} =(f_1, \ldots, f_m) $ an $m$-tuple of linear vector fields on some linear space $M$, then $B({\bf f})$, $B({\bf f})^{\left(\underset{{{\{m_1,m\}}}}{\bf t}{{, s_m}}, {\bf s} \right) }$ are also linear vector fields on $M$, for all ${\bf t}\in \rr^m$, ${\bf s} \in \rr^{m-1}$.   Denoting by ${\bf A} = (A_1, \ldots, A_m)$ the $m$-tuple of matrices associated in order to the components of ${\bf f}$ with respect to some fixed basis of $M$, then one can define in an obvious way matrices $B({\bf A})$, and
$B({\bf A})^{\left(\underset{{{\{m_1,m\}}}}{\bf t}{{, s_m}}, {\bf s} \right) }$ in such a way that they be the associated matrices of $B({\bf f})$ and $B({\bf f})^{\left(\underset{{{\{m_1,m\}}}}{\bf t}{{, s_m}}, {\bf s} \right) }$ with respect to that same basis of $M$.}
%\end{document}
One finds %(with the help of Computer Algebra System; we used Mathematica 9)
\begin{gather*}
[A_1, A_2]= \left(
\begin{array}{cc}
 {{-1}} & 0 \\
 {\phantom{-}}0 & {{1}} \\
\end{array}
\right),\\
[[A_1, A_2], A_3]= \left(
\begin{array}{cc}
 0 & 0 \\
 0 & 0 \\
\end{array}
\right), \\ 
[A_1, A_2]^{({{s_2}},s_1)} =
\left(
{{\begin{array}{cc}
 -2 s_1s_2-1 & 2s_2\left(s_1s_2+1\right)   \\
 -2 {s_1} & 2s_1s_2+1
\end{array}}}
\right){{,}}\\
\begin{array}{c}
[[A_1, A_2], A_3]^{(t_1,{{s_3}},s_1,s_2)} = \\
\begin{small}
\left(
{{\begin{array}{cc}
-2s_2^2t_1\left(s_1 + t_1 - s_2t_1^2 + 2s_1s_2t_1\right) & -2s_2e^{s_3}\left(s_2^2t_1^2 - 2s_1s_2^2t_1 - 2s_2t_1 + s_1s_2 + 1\right) \\
-2e^{-s_3}\left(- s_2^3t_1^4 + 2s_1s_2^3t_1^3 + 3s_1s_2^2t_1^2 + 2s_1s_2t_1 + s_1\right) &  2s_2^2t_1\left(s_1 + t_1 - s_2t_1^2 + 2s_1s_2t_1\right) \\
\end{array}}}
\right).
\end{small}
\end{array}
\end{gather*}
%Thus, in this case $[f_1,f_2]^{(t_2,s_1)} \neq [f_1, f_2]$.
By definition (or by formula \eqref{lie3int}), one gets\begin{small}
{{\begin{align*}
&(x,y)\Psi_{[[X_1, X_2],X_3] }^{(f_1, f_2,f_3)}(t_1,t_2,t_3) = (x, y)\\
&+\Big(t_1^3t_2^3\left(e^{-t_3} - 1\right)x + t_1t_2^2\left(e^{t_3} - 1\right)\left(t_1t_2 - 1\right)y, - t_1^2t_2e^{-t_3}\left(e^{t_3} - 1\right)\left(t_1^2t_2^2 + t_1t_2 + 1\right)x+t_1^3t_2^3\left(e^{t_3} - 1\right)y\Big)
\end{align*}}}
\end{small}
%\begin{multline}
%(x,y)\Psi_{[[X_1, X_2],X_3] }^{(f_1, f_2,f_3)}(t_1,t_2,t_3) =\\
%\frac{1}{12} \left(  {t_1}^3 {t_2}^3 e^{-{t_3}} \left(e^{{t_3}}-1\right)  \left({t_1}^2 {t_2}^2-12\right) x   + {t_1} {t_2}^2 e^{-{t_3}} %\left(e^{{t_3}}-1\right)  ({t_1} {t_2}-1) \left({t_1}^2 {t_2}^2-12\right) y, \right. \\
%{t_1}^2 {t_2}  \left(12 \left(e^{{t_3}}-1\right) ({t_1} {t_2} ({t_1} {t_2}+1)+1)-{t_1}^2 {t_2}^2 {t_3} ({t_1} {t_2} ({t_1} {t_2}+2)-6)\right) x \\  \left. -  {t_1}^2 {t_2}^2  \left({t_1}^3 {t_2}^3 {t_3}+{t_1}^2 {t_2}^2 {t_3}-4 {t_1} {t_2} \left(2 {t_3}+3 e^{{t_3}}-3\right)+6 {t_3}\right) y\right)
%\end{multline}
for all $(x, y) \in \rr^2$ and $t_1, t_2, t_3 \in \rr$. Hence, although $[[f_1,f_2],f_3]$ vanishes (identically), %$f_1, f_2, f_3$ do not commute
$\Psi_{[[X_1, X_2],X_3] }^{(f_1, f_2,f_3)}$ $(t_1,t_2,t_3)  \neq Id_M$, so that \eqref{adj2} cannot hold.
\end{example}

\subsection{Nonsmooth vector fields and set-valued Lie brackets}\label{Subsec-nonsmooth}

Let us conclude with a theme already mentioned in the Introduction. In \cite{RaSu01}  the following  notion  of {\it set valued} Lie bracket $[f,g]_{set}$ has been proposed  for locally Lipschitz continuous vector fields $f,g$:  for every $x\in M$, one lets
\bel{setbr}[f,g]_{set}(x) \equaldef co\left\{ \lim_{j\to \infty} [f,g](x_j)\,\,|\,\,  (x_j)_{j\in\nn}\subset DIFF(f)\cap DIFF(g),\,\, \lim_{j\to \infty}x_j=x \right\},
\eeq
where $co$ means {{the convex envelope}}, and $DIFF(f)\subset M$ and $DIFF(g)\subset M$ denote   the subsets where $f$ and $g$, respectively, are differentiable. By Rademacher's theorem,  these sets have full measure, so, in particular $DIFF(f)\cap DIFF(g)$ is dense in $M$.
The set-valued map $x\mapsto [f,g]_{set}(x)$ turns out to be upper  semicontinuous with compact, convex nonempty values.
In  \cite{RaSu01} this bracket  has been utilized to provide  a nonsmooth generalization of Chow-Rashevski's theorem. Successively it has been also used to prove  Frobenius-like and commutativity results for nonsmooth vector fields.
 Therefore, a natural issue might be a  generalization of this notion to formal brackets $B$ of degree $m\geq 3$ and vector fields that fail to be of class $C^B$. As mentioned in the {{I}}ntroduction, a mere iteration of   \eqref{setbr},  produces a (set-valued) bracket that is {\it too small} for various purposes, notably for asymptotic formulas.  For instance (see the example in \cite{RaSu07}*{Section~7}) one can find  a point $x\in M$ and vector fields $f,g$ with locally Lipschitz derivatives such that, setting $ h\equaldef [f,g]$,   {\it the  map $t\mapsto x\Psi_{[[X_1,X_2],X_3]}^{(f,g,h)}(t,t,t) -x$  is not $o(t^3)$}, while
$$
x[[f,g],h]_{set} =x[h,h]_{set} = 0{{.}} \footnote{ If $f, g$ where of class $C^2$ and hence h of class $C^1$, this map {\it would be $o(t^3)$}, namely
$$ x\Psi_{[[X_1,X_2],X_3]}^{(f,g,h)}(t,t,t) -x = x[[f,g],h]\cdot t^3 + o(t^3) =o(t^3).$$ }
$$
Our guess is that a suitable notion of iterated (set-valued) bracket should contain more tangent vectors than those  prescribed by definition \eqref{setbr} (namely the  limits  of sequences $ ([[f,g],h](x_j))_{j\in \nn}$ for $ (x_j)_{j\in\nn}\subset DIFF(h{{)}}$ with $x_j \to x$ as $j \to \infty$). More specifically, we think that the nested structure of a formal bracket $B$, and in particular, the recursive definition of integrating brackets,  suggests a new  notion of set-valued bracket giving rise to  asymptotic formulas   that might prove  useful for obtaining a higher order, {{nonsmooth}}, Chow-Rashevski type result.

\section*{Acknowledgements}

We wish to express  our deep gratitude to H{{{\'e}}}ctor J. Sussmann for his active help, made in particular of very inspiring  ideas, most of which at a quite advanced stage of development. Furthermore, we are very grateful to Rohit Gupta, who  read carefully the manuscript and helped us \red{significantly} in improving some quite misleading notation. We \red{are} also indebted {{to}} the anonymous referees for crucial comments on various aspects of the paper.

%%%%%%%%%%%%%%%%%%%%%%%%%%%%%%%%%%%%%%%%%%%%%%%%%%%%%%%%%%%%%%%%%%%%%%%%%%%%%%%%%%%%%%%%%%%%%%%%%%%%%%%%%%%%%%%%%%%

% \begin{remark}[A question]
%This question may arise naturally.
%Assume that the highest order derivatives of vector fields $f_1, \dots, f_r$ in Theorem~\ref{ChowIntro} needed &in order to compute
%\[
%B_1 ({ \bf f}_1 )(x_*), \ldots,  B_\ell ({ \bf f }_\ell )(x_*)
%\]
% are not continuous, but only defined almost everywhere, measurable and locally bounded. Let $x_*$ be a point %where all the said derivatives are defined, and the rank condition \eqref{rancond} is satisfied. Can we still %conclude that the thesis of Theorem~\ref{ChowIntro} holds true?  The answer is ``Yes'' but for its proof the %reader has to wait for the sequel of this paper.
%\end{remark}

\end{document}